\documentclass[a4paper,10pt]{article}

\usepackage{lipsum}
\usepackage{amsfonts, amsthm, amsmath,amssymb, amsopn}
\usepackage{graphicx}
\usepackage{epstopdf}
\usepackage{algorithmic}
\usepackage[numbers,sort]{natbib}
\usepackage[colorlinks]{hyperref}
\usepackage{mathtools}
\usepackage{mhchem}
\usepackage{bbm}

\usepackage{fullpage}

\numberwithin{equation}{section}

\newtheorem{theorem}{Theorem}[section]

\newtheorem{lemma}[theorem]{Lemma}

\newtheorem{assumption}{Assumption}[section]

\theoremstyle{definition}
\newtheorem{definition}{Definition}[section]
\newtheorem{example}{Example}[section]

\theoremstyle{remark}
\newtheorem{remark}{Remark}[section]

\title{Individual molecules dynamics in reaction network models}
\author{Daniele Cappelletti\footnotemark[1] \and Grzegorz A.\ Rempala\footnotemark[2]}
\date{}

\newcommand{\RR}{\mathbb{R}}
\newcommand{\ZZ}{\mathbb{Z}}
\newcommand{\Sp}{\mathcal{X}}
\newcommand{\C}{\mathcal{C}}
\newcommand{\Rc}{\mathcal{R}}
\newcommand{\G}{\mathcal{G}}
\newcommand{\A}{\mathcal{A}}
\newcommand{\track}{\mathcal{T}}
\newcommand{\tracksub}{\overline{\Sp}}
\newcommand{\trackS}{\tau}
\newcommand{\trackRc}{\widetilde{\Rc}}
\newcommand{\tr}[1]{\widetilde{#1}}
\newcommand{\cen}[1]{\overline{#1}}

\DeclarePairedDelimiter\floor{\lfloor}{\rfloor}

\DeclareMathOperator{\supp}{supp}
\DeclareMathOperator{\proj}{\pi}

\ifpdf
\hypersetup{
  pdftitle={Individual molecule dynamics},
  pdfauthor={D.~Cappelletti, and G.~Rempala}
}
\fi

\newcommand{\edit}[1]{{#1}}

\begin{document}

\footnotetext[1]{Politecnico di Torino, Torino, Italy (\texttt{daniele.cappelletti@polito.it}).}
\footnotetext[2]{The Ohio State University, Columbus, Ohio, USA (\texttt{rempala.3@osu.edu}).}

\maketitle

\begin{abstract} In a stochastic reaction network setting we consider the problem of tracking the fate of individual molecules. 
We show that   using   the  classical large volume limit results, we may approximate the dynamics of a single tracked molecule in a simple and computationally  efficient  way. We give examples on how this approach may be used  to obtain  various  characteristics of  single-molecule dynamics (for instance, the distribution of the number of infections in a single individual in the course of an epidemic or  the activity time of a single enzyme molecule). Moreover, we show how to approximate the overall dynamics of species of interest in the full system with a collection of independent single-molecule trajectories, and give explicit bounds for the approximation error in terms of the reaction rates. This approximation, which is well defined for all times, leads to an efficient and fully parallelizable simulation technique for which we provide some numerical examples. 
\end{abstract}

%

\section{Introduction}
Recent advances in modeling molecular systems, especially our improved ability to  track  individual proteins, and  the deluge  of data from the  observations of both molecular and macro system (think,  for instance, of the ongoing  COVID-19 pandemic),  have created new scientific challenges  of considering models of  very high resolution where the dynamics of a specific  bio-molecule or a particular individual   are of interest.  In general,  such 'agent-based' models   are known to be   computationally very costly, due to complex stochastic dynamics and highly noisy behavior of individual agents.  However,  it appears that, at least in some cases, simple  yet satisfactory approximation of individual molecular trajectory  may  be directly inferred with the  help  of a  classical approach of  stochastic chemical kinetics that assumes that all molecules or individuals are indistinguishable and consequently focuses only on  their aggregated counts.  As an example of  one such idea, originally proposed in \cite{caleb20} and latter expanded in \cite{wasiur20}, consider the   stochastic 'susceptible-infected'  ($SI$) chemical reaction network  where  a collection of $m+n$ molecules (or individuals) is  partitioned into  two types:  susceptible ($S$) and infected ($I$) with initially $n$ being of type   $S$ and remaining $m$ of type $I$.  The stochastic network evolves in time according to a Markov jump process that  counts the 'infection events', that is, the  interactions of  one molecule of $I$-type with  one molecule of $S$-type. Each such interaction  creates  a new molecule of  $I$-type  and  removes one of  $S$-type (equivalently, a molecule changes its type from $S$ to $I$). Accordingly, in the reaction network notation described below in Section~\ref{ssec:rnt} this model   may be represented as 
\begin{equation}\label{eq:SIint}
    S+I\ce{->}  2I.
\end{equation}
If the rate constant of the above reaction is $\beta/n$ and we assume the usual  mass action kinetics \cite{AK:2015}, it is well know that the  above stochastic reaction network satisfies the law of large numbers, in the sense that as $m,n\to \infty$ and $m/n\to \rho>0$ the surviving  proportion  $s_t$  of the $S$-type molecules  follows  the  logistic  equation that may be written in the form 
\begin{equation}\label{eq:s_t}
-\dot{s}_t/s_t=\beta (1+\rho-s_t)\qquad s_t(0)=1.
\end{equation} 
Consequently,  for $t\ge 0$ we have 
\begin{equation}\label{eq:surv}s_t= \frac{1+\rho}{1+\rho\exp(\beta(1+\rho)t)}.
\end{equation}  Thus,  from the viewpoint of  a single, randomly selected  $S$-type molecule,  the quantity $s_t$ defines a {\em survival function} describing the    limiting  {\em probability}  of  surviving beyond time $t>0$. The formula \eqref{eq:surv} led to the method of approximating the distribution of surviving molecules of $S$ dubbed  `dynamical survival analysis' (DSA)   described in  \cite{wasiur20} and applied recently to epidemic modeling \cite{mat21,prison21,di2022dynamic,hosp22,israel22}.  The  idea  is  further  illustrated in   Figure~\ref{fig:si} where the average of the Markov  process \eqref{eq:SIint} is compared to  the average of independent realizations of  single molecule dynamics (which may be efficiently calculated using modern  parallel computing capabilities).   Note 
 \eqref{eq:s_t} may be also interpreted as the equation for the  {\em hazard function} associated with $s_t$. This fact  has some relevance  for statistical inference, and is further exploited, for instance,   in  \cite{wasiur20,di2022dynamic}.
 \begin{figure}
     \centering
     \includegraphics[width=\textwidth]{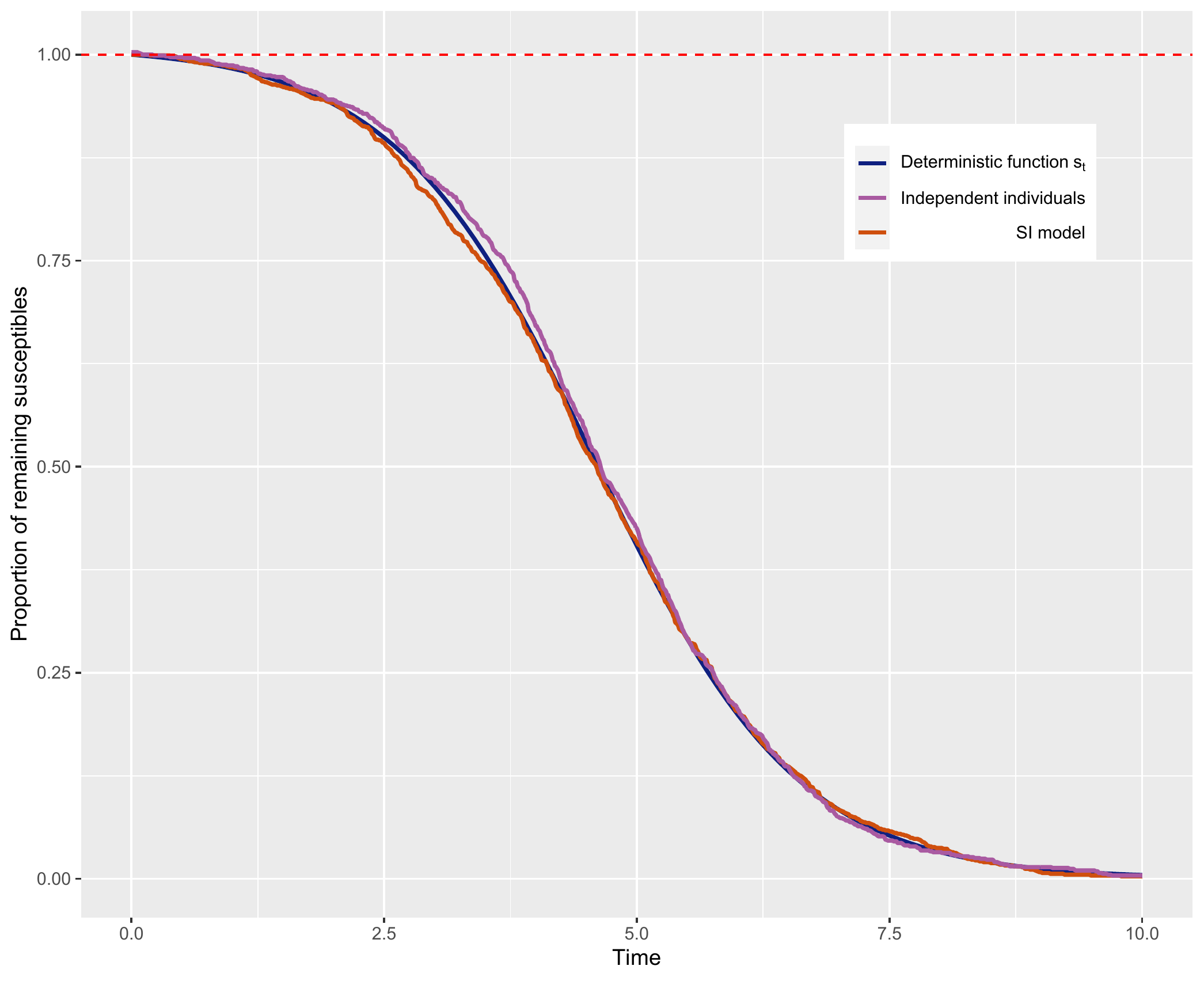}
     \caption{\textbf{Survival approximation  in the SI model.} The empirical trajectory of the proportion of the remaining $S$ molecules in the SI model described in \eqref{eq:SIint} as compared to the deterministic function $s_t$ defined in \eqref{eq:s_t} and the average of $1,000$ independent single trajectories of individuals who become infected according to $s_t$. For the simulation we considered $n=1,000$, $m=10$, $\beta=1$, and $\rho=0.01$.}
     \label{fig:si}
 \end{figure}
 
Beyond the simple $SI$ example, the DSA approach has been applied (mostly in the context of epidemics) only to a handful of reaction networks  representing the so-called one-directional transfer models \cite{caleb20}. In all such  networks  individual molecules can   only change their state in an ordered way, hence previously visited states are  no longer attainable (for instance in  the  $SI$ model a  molecule of $S$-type  can only change into $I$-type, but not vice-versa).   

In the  current paper we  formally expand  the survival function approach  for   tracking the fate of individual  molecules to  a  much broader class of networks,  including those where  molecules  can return to their previous stages.  A simple  example is  obtained by augmenting the $SI$  network with  the additional  reaction $I\to S$, leading  to the so-called $SIS$ model (which is of interest in epidemiology) discussed in more detail in Example~\ref{ex:SI} below.   To establish our results for such networks,   we explore a different   representation of the DSA approximation,  which  does not explicitly involve the  survival function.   Continuing with the $SI$ model example,    denote by $Y^i(t)$  the binary variable  that takes value 1 or 0    according to whether $i$-th  molecule is of type $S$ or $I$.  The limit dynamics of an  $i$-th individual molecule (initially of type $S$)  is then given by 
\begin{equation*}\label{eq:surv2} 
Y^i(t) =1- N^i\left(\beta\int_0^t Y^i(u)(1+\rho-s_u) du \right)
\end{equation*}
where  $N^i$  is the unit Poisson process tracking the transition of the $i$-th molecule from $S$-type  to $I$-type. Note that the argument of $N^i$  is the cumulative hazard corresponding to integral of the right-hand side of \eqref{eq:s_t} (see \cite{wasiur20}).  Such   Poisson process representation is of course completely equivalent to simply having 
the time of switching of the $i$-th molecule from $S$ to $I$ follow  the survival function  \eqref{eq:surv}, but it allows for a description of more complex scenarios than one-directional transfer models. For example, we will prove below that the limit dynamics of a single molecule in the $SIS$ model can be written as
\begin{equation*} 
Y^i(t) =1- N^i_1\left(\beta\int_0^t Y^i(u)(1+\rho-s_u) du \right)+ N^i_2\left(\kappa\int_0^t (1-Y^i(u)) du \right)
\end{equation*}
for independent and identically distributed unit-rate Poisson processes $N^i_1$ and $N^i_2$. Here, $\kappa$ is the rate constant of the reaction $I\to S$.

In this work we study the Poisson process representation of the DSA  approximation and give conditions under which it  describes a  single-molecule trajectory of the original network.  In particular, we explicitly derive error bounds of the DSA  approximation, in terms of the underlying reaction network rates. We illustrate via numerical examples how this novel technique could be useful to infer quantities pertaining to single-molecule dynamics (such as the distribution of the number of infections a single individual undergoes in a $SIS$ model, or the time a single enzyme spends in the bound state) in a computationally efficient way.

Further, we consider the  problem of comparing the dynamics of an original full reaction network with that of a collection of independent approximations of single-molecule trajectories and provide  explicit bounds on the error. Having the dynamics of the whole system approximated by a number of independent trajectories allows for computationally efficient simulation techniques, that are fully parallelizable. Moreover, since the DSA approximation is defined for all times, it does not suffer from the problem of exiting the state space as it is known to happen in other methods such as diffusion approximations or tau leaping \cite{diffusion, diffusion2, tau, tau2}. Finally, the independence of the single-molecule trajectories also allows for much simplified statistical inferential procedures. Such applications were already considered in the context of SIR networks in recent papers on the COVID-19 pandemic \cite{prison21,mat21,di2022dynamic,hosp22,israel22}. A thorough investigation of these techniques in general reaction networks is currently being conducted and will appear in a future work.

The paper is organized as follows: in Section~\ref{sec:background}  we provide  the necessary concepts pertaining to 
reaction network theory followed by  the result on the approximation in classical scaling in  Section~\ref{sec:cs}. In Section~\ref{sec:track}  we give a formal definition of what we refer to as \edit{ `status' of the molecules of interest}. In  Section~\ref{sec:results} we state our main results. In particular, in Section~\ref{sec:cssm} we give the theorem on the Poisson process representation of the DSA approximation for a single-molecule trajectory, and give examples of its applications in Section~\ref{sec:applications_single}.  Finally, in Section~\ref{sec:aggregate}  we  state the  result on the approximation of the original full network via  independent single-molecule trajectories, and give numerical examples. Proofs and explicit error bounds are given in the Appendix~\ref{sec:proof}.

\section{Background definitions}\label{sec:background}

\subsection{Notation}

 We denote by $\RR$, $\RR_{>0}$, and $\RR_{\geq0}$ the real, positive real, and non-negative real numbers, respectively. Similarly, we denote by $\ZZ$, $\ZZ_{\geq1}$, and $\ZZ_{\geq0}$ the real, positive real, and non-negative real numbers, respectively. Given a number $r\in\RR$, we denote by $|r|$ its absolute value, and by $\floor{r}$ the largest $m\in\ZZ$ such that $m\leq r$.
 
 Given a vectors $v\in\RR^n$, we denote its $i$th component by $v_i$, for all $1\leq i\leq n$. We further denote
 \begin{equation*}
     \|v\|_\infty=\max_{1\leq i\leq n} |v_i|\quad\text{and}\quad\floor{v}=(\floor{v_1}, \dots, \floor{v_n}).
 \end{equation*}
 Given two vectors $u,v\in\RR_{\geq0}^n$, we write
 \begin{equation*}
     u^v=\prod_{i=1}^m u_i^{v_i},
 \end{equation*}
 with the convention that $0^0=1$. We also write $u\geq v$ if the inequality holds component-wise. Furthermore, for any vector $v\in\ZZ_{\geq0}^n$, we write
 \begin{equation*}
     v!=\prod_{i=1}^m v_i!\,.
 \end{equation*}
 Given a set $A$, we denote its cardinality  by $\#A$ or, if it leads to no ambiguity, by $\vert A\vert$. 
We assume the reader is familiar with basic notions from stochastic process theory, such as the definition of continuous-time Markov chains and Poisson processes \cite{N:1998}.
 
 Consider a sequence of random variables $\{X_n\}_{n\in\ZZ_{\geq0}}$ and a random variable $X$, all defined on the same probability space and with values in a normed space $(E, \|\cdot\|)$. We say that $X_n$   converges in probability to $X$ if for all $\eta\in\RR_{>0}$
 \begin{equation*}
     \lim_{n\to\infty} P\left(\|X_n-X\|>\eta\right)=0.
 \end{equation*}
 
 Given a topological space $E$ we will denote by $D_E[0,T]$ the set of right-continuous left-bounded functions defined from $[0,T]$ to $E$, endowed with the Skorokhod $J_1$ topology. In particular, we say that the sequence of processes $\{X_n\}$ with sample paths in $D_E[0,T]$ converges in probability to the process $X$ (or simply that $X_n$ converges in probability to $X$) if the Skorokhod distance between $X_n$ and $X$ converges to 0 in probability (for more details, see for example \cite[Chapter 3]{EK:1986}).
 
\subsection{Stochastic reaction networks} \label{ssec:rnt}


 A \emph{reaction network} is a triple $\G=\{\Sp, \C, \Rc\}$, where (a) $\Sp$ is an ordered finite sequence of $d$ symbols, called \emph{species}; (b) $\C$ is a finite set of linear combinations of species over $\ZZ_{\geq0}$, called \emph{complexes}; (c) $\Rc$ is a finite set of elements of $\C\times\C$, called \emph{reactions}.
 We assume that no element of the form $(y,y)$ is in $\Rc$, for any complex $y$, even though our results do not depend on this assumption. Following the usual notation of reaction network Theory, we further denote a reaction $(y, y')\in\Rc$ by $y\to y'$. We finally assume that each complex appears in at least one reaction, and that each species has a positive coefficient in at least one complex. Under this assumption and up to ordering of the set of species, a reaction network is uniquely determined by the set $\Rc$, or equivalently by the directed graph $(\C, \Rc)$, called \emph{reaction graph}. As an example, consider the reaction graph
 \begin{equation}\label{example}
 A+B\ce{<=>}2B,\quad B\ce{->} C.
 \end{equation}
 In this case, the associated species are $A$, $B$, and $C$, $\C=\{A+B, 2B, B, C\}$, and $\Rc=\{A+B\to 2B, 2B\to A+B, B\to C\}$.
 
 In this paper we will implicitly identify $\RR^{\vert \Sp\vert}$ with $\RR^d$, and therefore each $S\in\Sp$ with a canonical basis vector of $\RR^d$. With this in mind, the complexes are linear combination of species and can be therefore considered as vectors in $\ZZ_{\geq0}^d$. As an example, if we order the species of \eqref{example} alphabetically, then the complex $A+B$ can be associated with the vector $(1,1,0)$, the complex $2B$ can be associated with $(0,2,0)$, the complex $C$ with $(0,0,1)$, and so on.
 We will tacitly use the identification of complexes with integer vectors throughout the paper. Moreover, for each vector $v\in\RR^d$ and for each species $S\in\Sp$ we denote by $v_S$ the entry of $v$ related to the canonical vector associated with $S$. We further define the \emph{support} of $v$ as $\supp(v)=\{S\in\Sp\,:\,v_{S}>0\}$. As an example, with the species of \eqref{example} alphabetically ordered, the support of $(1,1,0)$ is $\{A,B\}$, the support of $(0,2,0)$ is $\{B\}$, and so on.
 
 
 Deterministic and stochastic dynamical systems can be associated with a reaction network. The stochastic model is usually utilized when few individuals are present, so the stochastic component of the dynamic behaviour should not be ignored. In this case, the time evolution of the number of individuals of the different species is considered, for certain given propensities of the reactions to occur, and modeled via a continuous time Markov chain. More precisely, a \emph{stochastic kinetics} for a reaction network $\G$ is a correspondence between a reaction $y\to y'$ and a \emph{rate function} $\lambda_{y\to y'}\colon \ZZ_{\geq0}^d\to \RR_{\geq0}$, such that $\lambda_{y\to y'}(x)>0$ only if $x\geq y$. A \emph{stochastic reaction system} is a continuous time Markov chain $\{X(t)\,:\,t\geq0\}$ with state space $\ZZ_{\geq0}^d$ and transition rates from a state $x$ to a state $x'$ defined by
 \begin{equation*}
  q(x,x')=\sum_{\substack{y\to y'\in\Rc\\ y'-y=x'-x}}\lambda_{y\to y'}(x).   
 \end{equation*}
 The associated generator is defined by
 \begin{equation*}
     Af(x)=\sum_{y\to y'\in\Rc} \lambda_{y\to y'}(x)\Big(f(x+y'-y)-f(x)\Big)
 \end{equation*}
 for any function $f\colon \ZZ_{\geq0}^d\to \RR$ and any $x\in\ZZ_{\geq0}^d$. Equivalently, the process $X$ can be described by
 \begin{equation*}
     X(t)=X(0)+\sum_{y\to y'\in\Rc} (y'-y)N_{y\to y'}\left(\int_0^\infty \lambda_{y\to y'} (X(s))ds\right),
 \end{equation*}
 where the processes $\{N_{y\to y'}\}_{y\to y'\in\Rc}$ are independent unit-rate Poisson processes. For more details on this representation, we refer to \cite{AK:2015} or \cite[Chapter 6]{EK:1986}.
 
 In the deterministic setting, the concentration of the different species are assumed to evolve according to an ordinary differential equation (ODE). Specifically,
 a \emph{deterministic kinetics} for a reaction network $\G$ is a correspondence between the reactions $y\to y'$ and the \emph{rate function} $\lambda_{y\to y'}\colon \RR_{\geq0}^d\to \RR_{\geq0}$, such that $\lambda_{y\to y'}(x)>0$ only if $x_i> 0$ whenever $y_i>0$. A \emph{deterministic reaction system} is the solution to the ordinary differential equation
 \begin{equation}\label{eq:drn}
  \frac{d}{dt} Z(t)=\sum_{y\to y'\in\Rc}(y'-y)\lambda_{y\to y'}(x).
 \end{equation}

 While our results hold in a more general scenario, all the simulations we show assume \emph{mass-action kinetics}, a popular choice of kinetics derived by the assumption that all the \edit{species molecules} are well-mixed in the available volume \cite{AK:2015}. Specifically, a stochastic reaction system is a \emph{stochastic mass-action system} if for every reaction $y\to y'\in\Rc$ we have
 \begin{equation*}
     \lambda_{y\to y'}(x)=\kappa_{y\to y'}\frac{x!}{(x-y)!}\mathbbm{1}_{\{x\geq y\}},
 \end{equation*}
 for some positive constant $\kappa_{y\to y'}$ called \emph{rate constant}. Similarly, a deterministic reaction system is a \emph{deterministic mass-action system} if for every reaction $y\to y'\in\Rc$ we have
 \begin{equation*}
     \lambda_{y\to y'}(x)=\kappa_{y\to y'}x^y,
 \end{equation*}
 for some positive constant $\kappa_{y\to y'}$ also called \emph{rate constant}.
 
\section{Classical scaling}\label{sec:cs}

Consider a reaction network $\G=\{\Sp, \C, \Rc\}$, and a family of stochastic kinetics $\{\lambda^V_{y\to y'}\,:\, y\to y'\in\Rc\}$ indexed by $V$. Let $X^V$ denote the associated continuous time Markov chain. $V$ should be thought to as a parameter expressing the volume, or the magnitude of the number of the present individuals. Under the following technical but reasonable assumption the classical scaling of \cite{kurtz1972, EK:1986} holds:
\begin{assumption}\label{ass:classical_limit}
 We assume that for any reaction $y\to y'\in \Rc$ there exists a locally Lipschitz function $\lambda_{y\to y'}\colon\RR^d_{\geq0}\to\RR^d_{\geq0}$ such that for any compact set $K\subset\RR^d_{\geq0}$ we have
 \begin{equation*}
  \lim_{V\to \infty}\sup_{z\in K}\left|\frac{\lambda^V_{y\to y'}(\floor{Vz})}{V}-\lambda_{y\to y'}(z)\right|=0.   
 \end{equation*}
 \end{assumption}

\begin{theorem}\label{thm:classical}
 Assume that   Assumption \ref{ass:classical_limit} holds. Furthermore, assume that the random variables $X^V(0)/V$ converge in probability to a constant $z^*$ as $V$ goes to infinity. Finally, let $\{Z(t)\,:\,t\geq0\}$ be the unique solution to \eqref{eq:drn} with $Z(0)=z^*$. Then, for any $\varepsilon>0$ and any $T>0$
 \begin{equation*}
     \lim_{V\to\infty}P\left(\sup_{t\in[0,T]}\left\|\frac{X^V(t)}{V}-Z(t)\right\|_\infty>\varepsilon\right)=0.
 \end{equation*}
\end{theorem}

Note that the distribution of the fate of a single molecule is not given, since the classical scaling concerns average dynamics. The goal of this paper is to address this issue, by providing a technique to simulate an approximation of the time evolution of a single observable species, as described in the next section.

\section{Molecular status}\label{sec:track}

\edit{We consider the problem of tracking the fate of an individual molecule through its  transformations into different species in  a certain stochastic  reaction network. For instance, we could be interested in the change in status of a single tracked individual of type $S$ in the SI model, discussed in the Introduction. To introduce a more general scenario where it is desirable to track the time evolution of different \emph{parts} of a species molecule, we give the following example.
\begin{example}\label{ex:MM1}
 Consider the following reaction network, depicting a Michaelis-Menten mechanism where the product protein and the enzyme can spontaneously transform into each other:
  \begin{equation}
  E+S\ce{<=>}C\ce{->}E+P,\quad P\ce{<=>} E.
 \end{equation}
 In particular, the complex $C$ represents a molecule of substrate $S$ and enzyme $E$ bound together. When the bond is broken, the molecule of enzyme is released while the molecule of substrate is either released or transformed into the product $P$. Suppose we want to keep track of the history of a molecule of substrate $S$. If we were dealing with a classic Michaelis-Menten kinetics, i.e. without the reactions $P\rightleftharpoons E$, then we could simply consider $S$, $C$, and $P$ as status for the tracked molecule, corresponding to unbound substrate, bound substrate, and product, respectively. Since the reactions $P\rightleftharpoons E$ are present, if we want to keep track of the fate of a molecule of substrate we need to take into account the fact that it can ultimately (via complex, then protein) be transformed into an enzyme, so $E$ becomes a possible status of the molecule. We now need to differentiate between the parts of a complex molecule of $C$ that a molecule of $E$ and a molecule of $S$ get transformed into by the reaction $E+S\to C$. The part of a (complex) molecule of $C$ that a molecule of $E$ gets transformed into will become a free enzyme again via the reaction $C\to E+P$, while the part a molecule of $C$ that a molecule of $S$ gets transformed into will become a molecule of product $P$ via $C\to E+P$. Here and below by ``part of a molecule" we  mean a part of a molecular complex rather then  one of  atoms comprising the specific molecule. To formally describe such dynamics  we consider $\{E,S,P,C_E,C_S\}$ as the set of molecular status, where $C_E$ denotes we are tracking a molecule of $E$ bound in the complex $C$, and $C_S$ denotes we are tracking a molecule of $S$ bound in $C$. Note that some status correspond to species, some other status do not. In order to avoid any notational confusion between the potentially different sets of chemical species and molecule status,  we adopt the convention of using tildes for status. In the present example, we will denote the set of tracked molecule status by $\{\tr{E},\tr{S},\tr{P}, \tr{C}_E, \tr{C}_S\}$.
\end{example}

Based on the above example, we see that the molecules whose dynamics we want to follow may or may not correspond to a subset of the chemical species $\Sp$. To deal with this general setting, we formally represent \emph{status} by a set $\track$ of symbols endowed with a function $\sigma\colon \track\to \Sp\cup\{0\}$ which links every status with its corresponding species in $\Sp$. For instance, in Example~\ref{ex:MM1} above we will choose $\sigma(\tr{S})=S$ and $\sigma(\tr{C}_E)=C$. Note that the number of status defined in this way can be less than, equal to, or larger than the number of species. A molecule that changes its status with time will be referred to as a \emph{tracked molecule}.}

The set $\track$ needs to include the special state $\Delta$ to denote the potential degradation of the tracked molecule, and we set $\sigma(\Delta)=0$. To simplify the notation, for all $x,y\in \ZZ_{\geq0}^d$ and $\trackS\in\track\setminus\{\Delta\}$ we denote by $\theta_{y}(\trackS,x)$ the probability that a certain molecule of species $\sigma(\trackS)$ is chosen if $y_{\sigma(\trackS)}$ molecules are uniformly drawn out of $x_{\sigma(\trackS)}$ molecules of $\sigma(\trackS)$ available. Specifically,
\begin{equation*}
\theta_{y}(\trackS,x)=\begin{cases}
               \frac{\binom{x_{\sigma(\trackS)}\;-1}{y_{\sigma(\trackS)}-1}}{\binom{x_{\sigma(\trackS)}}{y_{\sigma(\trackS)}}}=\frac{y_{\sigma(\trackS)}}{x_{\sigma(\trackS)}}&\text{if }x_{\sigma(\trackS)}\geq y_{\sigma(\trackS)}\geq1\\
               0&\text{otherwise}
              \end{cases}. 
\end{equation*}
For completeness, we define $\theta_y(\Delta,x)=0$. Finally, note that in reactions such as $2A\to B+C$ we can imagine a molecule of $A$ is transformed into a molecule of $B$, while the other molecule of $A$ turns into a molecule of $C$. If we are tracking the fate of $A$ molecules and the reaction $2A\to B+C$ occurs, it is reasonable to assume the molecule we are tracking has a 50\% change of turning into a molecule of $B$, and a 50\% change of becoming a molecule of $C$. We denote these probabilities with $p_{2A\to B+C}(A,B)$ and $p_{2A\to B+C}(A,C)$, respectively, and in general allow for different value choices, as along as $p_{2A\to B+C}(A,B)+p_{2A\to B+C}(A,C)=1$. The definition of \emph{tracking stochastic reaction system} in the most general setting is below.
\begin{definition}[Tracking stochastic reaction system]\label{def:srswts}
 Let $\G=\{\Sp, \C, \Rc\}$ be a reaction network. Consider a family of stochastic kinetics $\{\lambda^V_{y\to y'}\,:\, y\to y'\in\Rc\}$ indexed by $V$, and let $X^V$ denote the associated continuous time Markov chains. Let $\track$ be a set of status. We define the \emph{tracking stochastic reaction system} as the continuous-time Markov chain $(Y^V, X^V)$ with state space $\track\times\ZZ_{\geq0}^d$ and transition rates
 \begin{align*}
  q\Big((\Delta,x), (\trackS',x')\Big)&=\mathbbm{1}_{\{\trackS'\}}(\Delta)\sum_{\mathclap{\substack{y\to y'\in\Rc\\ y'-y=x'-x}}}\lambda^V_{y\to y'}(x)\\
  \intertext{and for all $\trackS\neq \Delta$}
  q\Big((\trackS,x), (\trackS',x')\Big)&=\sum_{\mathclap{\substack{y\to y'\in\Rc\\ y'-y=x'-x}}}\Big((1-\theta_{y}(\trackS,x))\mathbbm{1}_{\{\trackS'\}}(\trackS)+\theta_{y}(\trackS,x)p_{y\to y'}(\trackS,\trackS')\Big)\lambda^V_{y\to y'}(x),
 \end{align*}
 where for all reactions $y\to y'\in\Rc$ the following holds:
 \begin{itemize}
  \item for any $\trackS\in\track,\trackS'\in\track\cup\{\Delta\}$ we have $0\leq p_{y\to y'}(\trackS,\trackS')\leq 1$;
  \item $p_{y\to y'}(\trackS,\trackS')=0$ whenever $\sigma(\trackS)\notin\supp(y)$ or $\sigma(\trackS')\notin\supp(y')\cup\Delta$;
  \item if $\sigma(\trackS)\in\supp(y)$ then
  \begin{equation*}
      \sum_{\trackS'\in\track\,:\,\sigma(\trackS')\in\supp(y')\cup\Delta}p_{y\to y'}(\trackS,\trackS')=1.
  \end{equation*}
 \end{itemize}
\end{definition}

In the above definition, the usual stochastic reaction system is coupled with the fate of a single tracked molecule: a molecule in \edit{status} $\trackS$ can transform whenever a reaction $y\to y'$ occurs, with a probability given by $\theta_{y}(\trackS,\trackS')$. By definition, the quantity $\theta_{y}(\trackS,\trackS')$ denotes precisely the probability that the tracked molecule takes part in the reaction $y\to y'$, assuming that the reacting molecules are uniformly chosen among those present. If that happens, the new state of the tracked molecule is drawn according to the probability distribution $\{p_{y\to y'}(\trackS,\trackS')\}_{\trackS'\in\supp(y')\cup\Delta}$ (see Example~\ref{ex:dimerization} for a case where this distribution is non-trivial). If the tracked molecule is irreversibly degraded, its \edit{status} becomes $\Delta$ and cannot be \edit{further} changed. \edit{In what follows, we will sometimes identify the state space of $Y^V$, given by $\track$, with the canonical basis of $\RR^{\vert\track\vert}$, similarly to how complexes are implicitly identified with vectors in $\RR^d$.}

The only technical requirement to have a \edit{tracking stochastic reaction system is establishing a rule on the status changes of the tracked molecules involved in a reaction. Mathematically, this can always be done. For instance, choose $\track=\Sp$ and let $\sigma$ be the identity. Consider a reaction $y\to y'$. If $\|y\|_1\leq \|y'\|_1$, then an injective map from the molecules consumed to the molecules created can be defined, giving a rule for molecular status change. If instead $\|y\|_1> \|y'\|_1$, then any molecule consumed can be either injectively mapped to a molecule created, or mapped to the cemetery status $\Delta$. Hence, formally the requirements of Definition~\ref{def:srswts} can always be satisfied for some choices of $\track$ and $\sigma$. However, care needs to be exercised if we want status changes to reflect physical properties of the system (see Example~\ref{ex:MM1}).}


\begin{remark}
 The generator of a tracking stochastic reaction system, as defined in Definition~\ref{def:srswts}, is given by
 \begin{equation*}
     \mathcal{A}f(\Delta,x)=\sum_{y\to y'\in\Rc}\lambda^V_{y\to y'}(x)\Big(f(\Delta,x+y'-y)-f(\Delta,x)\Big)
 \end{equation*}
 and for $\trackS\neq\Delta$
 \begin{multline*}
  \mathcal{A}f(\trackS,x)=\sum_{y\to y'\in\Rc}(1-\theta_{y}(\trackS,x))\lambda^V_{y\to y'}(x)\Big(f(\trackS,x+y'-y)-f(\trackS,x)\Big)\\
  +\sum_{y\to y'\in\Rc}\sum_{\trackS'\in\supp(y')\cup\Delta}\theta_{y}(\trackS,x)p_{y\to y'}(\trackS,\trackS')\lambda^V_{y\to y'}(x)\Big(f(\trackS',x+y'-y)-f(\trackS,x)\Big),
 \end{multline*}
 for all functions $f\colon(\track)\times\ZZ_{\geq0}^d\to\RR$.
\end{remark}

\begin{example}\label{ex:SI}
Consider the SI reaction network described in \eqref{eq:SIint}, which we repeat here for convenience:
 \begin{equation}\label{eq:SI}
  S+I\ce{->}2I. 
 \end{equation}
 In this case, we are interested in describing the history of  susceptible individuals who become infected. The set of status is therefore $\track=\{\tr{S},\tr{I}\}$ with $\sigma(\tr{S})=S$ and $\sigma(\tr{I})=I$. Furthermore, we choose the probabilities $p_{S+I\to2I}(\tr{S},\tr{I})=1$ and $p_{S+I\to2I}(\tr{I},\tr{I})=1$. Alternatively, one can simply consider $\track=\{\tr{S}\}$, with the understanding that whenever a susceptible individual gets infected we consider it as irreversibly degraded, and its state becomes $\Delta$. In this case, $p_{S+I\to2I}(\tr{S},\Delta)=1$.
 
 The state of single individuals can be tracked also in the more complex model
  \begin{equation}\label{eq:SIS}
  S+I\ce{->}2I,\quad I\ce{->} S.
 \end{equation}
 Here, the set of status is $\{\tr{S},\tr{I}\}$, with $\sigma(\tr{S})=S$ and $\sigma(\tr{I})=I$, and the transformation probabilities are $p_{S+I\to2I}(\tr{S},\tr{I})=1$, $p_{S+I\to2I}(\tr{I},\tr{I})=1$, $p_{I\to S}(\tr{I},\tr{S})=1$. Here, relevant questions on the fate of a single individual could concern, for example, the number of infections it undergoes in a given time, or after how long the $n$th infection occurs. We can even extend the model to include migrations, and obtain
 \begin{equation}\label{eq:SIS_migration}
  S+I\ce{->}2I,\quad I\ce{->} S,\quad 0\ce{<=>}S,\quad 0\ce{<=>}I.
 \end{equation}
 In this case, it is natural to assume $p_{S\to0}(\tr{S},\Delta)=1$ and $p_{I\to0}(\tr{I},\Delta)=1$. Relevant questions could involve, for example, the average number of infection a susceptible individual undergoes before migrating.
\end{example}

\begin{example}\label{ex:dimerization}
 Consider the following reaction network, where a protein $P$ promotes its own phosphorylation:
 \begin{equation}
  2P\ce{->}P+P^*,\quad P^*\ce{->} P, P\ce{->}0.
 \end{equation}
 Here, we may assume we are interested in observing the dynamics of a molecule of protein $P$. Hence, the set of status is $\{\tr{P},\tr{P}^*\}$ with $\sigma(\tr{P})=P$ and $\sigma(\tr{P}^*)=P^*$. It is natural to assume that the two molecules of $P$ involved in the reaction $2P\to P+P^*$ have the same probability of being phosphorylated or serving as the reaction catalyst. Hence, $p_{2P\to P+P^*}(\tr{P},\tr{P})=p_{2P\to P+P^*}(\tr{P},\tr{P}^*)=1/2$. The other transformation probabilities are given by $p_{P^*\to P}(\tr{P}^*,\tr{P})=1$ and $p_{P\to0}(\tr{P},\Delta)=1$.
\end{example}

\begin{example}\label{ex:MM}
 \edit{Consider the reaction network of Example~\ref{ex:MM1}:
  \begin{equation}
  E+S\ce{<=>}C\ce{->}E+P,\quad P\ce{<=>} E.
 \end{equation}
 We consider the set of status $\{\tr{E},\tr{S},\tr{P}, \tr{C}_E, \tr{C}_S\}$, as described above. In this case the function $\sigma$ associates every status of the molecules with the chemical species they are part of}: $\sigma(\tr{E})=E$, $\sigma(\tr{S})=S$, $\sigma(\tr{P})=P$, $\sigma(\tr{C}_E)=C$, and $\sigma(\tr{C}_S)=C$. The transformation probabilities are given by
 \begin{center}
 \begin{tabular}{lclcl}
 $p_{E+S\to C}(\tr{E},\tr{C}_E)=1$ && $p_{C\to E+S}(\tr{C}_E,\tr{E})=1$ && $p_{C\to E+P}(\tr{C}_E,\tr{E})=1$ \\
 $p_{E+S\to C}(\tr{S},\tr{C}_S)=1$ && $p_{C\to E+S}(\tr{C}_S,\tr{S})=1$ && $p_{C\to E+P}(\tr{C}_S,\tr{P})=1$ \\
 $p_{P\to E}(\tr{P},\tr{E})=1$ && $p_{E\to P}(\tr{E},\tr{P})=1$
 \end{tabular}
 \end{center}

\end{example}

\begin{remark}
 The interpretation of a tracking stochastic reaction system is that of a regular stochastic reaction system with the subsequent tranformations of a given particle being tracked. If the initial state $Y^V(0)$ of the tracked molecule is not present in the initial $X^V(0)$, that is if $X^V_{\sigma(Y^V(0))}(0)=0$, then the initial condition of $(Y^V, X^V)$ is not consistent with the interpretation of the process. The process $(Y^V, X^V)$ is still well-defined and its evolution can be studied, but its interpretation is no longer valid. In order to obtain meaningful results, we therefore tacitly assume that $X^V_{\sigma(Y^V(0))}(0)>0$, even if we do not require it formally.
\end{remark}

\subsection{Representation as a regular stochastic reaction network}\label{sec:cast} 

In this section we show how a  tracking stochastic reaction system $(Y^V, X^V)$ can be realized as a regular stochastic reaction system with species set given by $\track\sqcup\Sp$, where $\sqcup$ denotes a disjoint union. In particular, the state space is $\ZZ^{|\track|}_{\geq0}\times\ZZ^d_{\geq0}$, where for convenience we consider the first coordinates to refer to $\track$, and the rest to the species of the original process $\Sp$. We denote by $(\tr{x},x)$ a generic state in $\ZZ^{|\track|}_{\geq0}\times\ZZ^d_{\geq0}$. Consider the set of reactions $\Rc\cup\trackRc$ where
\begin{equation*}
    \trackRc=\{\trackS+y\to \trackS'+y'\,:\,y\to y'\in\Rc, \trackS,\trackS'\in\track\text{ and }p_{y\to y'}(\trackS, \trackS')>0\}
\end{equation*}
and endow them with the following reaction rates:
\begin{align*}
 \lambda^V_{y\to y'}(\tr{x},x)&=\sum_{\trackS\in\track}\tr{x}_{\trackS}(1-\theta_{y}(\trackS,x))\lambda^V_{y\to y'}(x)\\
 \lambda^V_{\trackS+y\to \trackS'+y'}(\tr{x},x)&=\tr{x}_{\trackS}\theta_{y}(\trackS,x)p_{y\to y'}(\trackS,\trackS')\lambda^V_{y\to y'}(x).
\end{align*}
Note that the second component of the process has the same transitions as $X^V$, with exactly the same rates. Hence, we can safely denote the process associated with the above stochastic reaction network by $(\tr{Y}^V, X^V)$.
Note that the quantity $\sum_{\trackS\in\track}\tr{x}_{\trackS}$ is conserved by all possible transitions. Hence, if we consider an initial condition $(\tr{Y}(0), X(0))$ with $\sum_{\trackS\in\track}\tr{Y}_{\trackS}(0)=1$, then at any time point $t$ exactly one entry of the vector $\tr{Y}(t)$ is 1, and the other entries are zero. It follows that there is a bijection between the possible values of $\tr{Y}$ and $\track$, given by the function $\supp(\tr{Y}(t))$. In this case, by identifying status with vectors of the canonical basis of $\RR^{|\track|}$ as already done in the paper for the species in $\Sp$, the transition rates can be equivalently written as
\begin{align*}
 \lambda^V_{y\to y'}(\tr{x},x)&=\sum_{\trackS\in\track}\mathbbm{1}_{\{\trackS\}}(\tr{x})(1-\theta_{y}(\trackS,x))\lambda^V_{y\to y'}(x)\\
 \lambda^V_{\trackS+y\to \trackS'+y'}(\tr{x},x)&=\mathbbm{1}_{\{\trackS\}}(\tr{x})\theta_{y}(\trackS,x)p_{y\to y'}(\trackS,\trackS')\lambda^V_{y\to y'}(x),
\end{align*}
Hence, if $\sum_{\trackS\in\track}\tr{Y}_{\trackS}(0)=1$ then the transitions and the rates of $(Y^V, X^V)$ and $(\tr{Y}^V, X^V)$ coincide, and $(Y^V, X^V)$ can be therefore realized as a stochastic reaction network with an appropriate initial condition. In particular, we can write
\begin{align}\label{eq:X_kurtznotation}
 X^V(t)&=X^V(0)+\sum_{y\to y'\in\Rc}(y'-y)N_{y\to y'}\left(\int_0^t\lambda^V_{y\to y'}(X^V(s))ds\right)\\
 \label{eq:Y_kurtznotation}
 Y^V(t)&=Y^V(0)+\sum_{y+\trackS\to y'+\trackS'\in\trackRc}(\trackS'-\trackS)N_{y+\trackS\to y'+\trackS'}\left(\int_0^t \lambda^V_{\trackS+y\to \trackS'+y'}(Y^V(s),X^V(s))ds\right)
\end{align}
where $N_r$ for $r\in\Rc\cup\trackRc$ are independent unit-rate Poisson processes. Note that with the above writing, all the processes in the set $\{(Y^V,X^V)\}_{V\in\ZZ_{\geq1}}$ can be defined on the same probability space.

\section{Results}\label{sec:results}

In this section we state our  main results  and illustrate   their applications.

\subsection{Classical scaling for the fate of a single molecule}\label{sec:cssm}

In this section we state a law of large number for the process $Y^V$. In order to do this, we consider a family of tracking stochastic reaction systems $(Y^V, X^V)$, with $V$ varying in the integer numbers greater than one. We then assume that Assumption \ref{ass:classical_limit} is satisfied for some locally Lipschitz functions $\lambda_{y\to y'}$, and denote by $Z$ the solution to \eqref{eq:drn}. Hence, we know by Theorem~\ref{thm:classical} that $V^{-1}X^V$ will converge to $Z$ path-wise with the uniform convergence topology over compact intervals of time, for $V$ going to infinity. 

In this section we express $(Y^V, X^V)$ by means of independent unit-rate Poisson processes, as in \eqref{eq:X_kurtznotation} and \eqref{eq:Y_kurtznotation}. With the notation introduced in the previous section in mind, we have the following first technical result:

\begin{lemma}\label{lem:convergence_hatlambda}
 Assume that Assumption \ref{ass:classical_limit} holds. Then, for any $\trackS+y\to \trackS'+y'\in\trackRc$, any $w\in\track$, and any compact set $K\subset\RR^d_{>0}$ we have
 \begin{equation}\label{skdfhdfkghrieuvh}
  \lim_{V\to \infty}\sup_{z\in K}\left|\lambda^V_{\trackS+y\to \trackS'+y'\in\trackRc}(w,\floor{Vz})-\lambda_{y\to y'}(w,z)\right|=0,
 \end{equation}
 where the function $\lambda_{\trackS+y\to \trackS'+y'}\colon \track\times\RR_{\geq0}^d$ is defined as
 \begin{equation*}
     \lambda_{\trackS+y\to \trackS'+y'}(w,z)=\mathbbm{1}_{\{w\}}(\trackS)p_{y\to y'}(\trackS,\trackS')y_{\sigma(\trackS)}\frac{\lambda_{y\to y'}(z)}{z_{\sigma(\trackS)}}
 \end{equation*}
 if both $z_{\sigma(\trackS)}$ and $y_{\sigma(\trackS)}$ are positive, and zero otherwise. Moreover, the function $\lambda_{\trackS+y\to \trackS'+y'}$ is locally Lipschitz if restricted to $\track\times\RR_{>0}^d$.
\end{lemma}
\begin{proof}
 If $y_{\sigma(S)}=0$, then both $\lambda^V_{\trackS+y\to \trackS'+y'\in\trackRc}$ and $\lambda_{y\to y'}$ are constantly zero, hence \eqref{skdfhdfkghrieuvh} holds. If $y_{\sigma(S)}$ is positive, then for all $z\in K$ we have
 \begin{equation*}
  \left|\lambda^V_{\trackS+y\to \trackS'+y'\in\trackRc}(w,\floor{Vz})-\lambda_{y\to y'}(w,z)\right|=
  \mathbbm{1}_{\{w\}}(\trackS)p_{y\to y'}(\trackS,\trackS')
  \left|
  \theta_{y}(\trackS,\floor{Vz})\lambda^V_{y\to y'}(\floor{Vz})-
  y_{\sigma(S)}\frac{\lambda_{y\to y'}(z)}{z_{\sigma(S)}}
  \right|
 \end{equation*}
  Let $m=\min_{z\in K}z_{\sigma{\trackS}}$, which is positive because $K$ is a compact set contained in $\RR^d_{>0}$. If $V$ is large enough such that $Vm>y_{\sigma{\trackS}}$ then
 \begin{equation*}
  \left|\lambda^V_{\trackS+y\to \trackS'+y'\in\trackRc}(w,\floor{Vz})-\lambda_{y\to y'}(w,z)\right|=
  \mathbbm{1}_{\{w\}}(\trackS)p_{y\to y'}(\trackS,\trackS')y_{\sigma(S)}
  \left|
  \frac{\lambda^V_{y\to y'}(\floor{Vz})}{V\cdot(\floor{Vz_{\sigma(\trackS)}}/V)}-
  \frac{\lambda_{y\to y'}(z)}{z_{\sigma(S)}}
  \right|
 \end{equation*}
 Hence, \eqref{skdfhdfkghrieuvh} follows from Assumption~\ref{ass:classical_limit} and
 \begin{equation*}
     \max_{z\in K}\left|\frac{\floor{Vz_{\sigma(\trackS)}}}{V}-z_{\sigma(\trackS)}\right|\leq \frac{1}{V}.
 \end{equation*}
 To conclude the proof, we only need to show that $\lambda_{\trackS+y\to \trackS'+y'}$ restricted to $\track\times\RR_{>0}^d$ is locally Lipschitz. However, this follows from it being the product (up to multiplication by a constant) of the two locally Lipschitz functions $z\mapsto1/z_{\sigma(\trackS)}$ and $\lambda_{y\to y'}$.
\end{proof}

 The main goal of this section is to prove a classical scaling limit for a single-molecule trajectory. To this aim, define the process $Y$ by
\begin{equation}
 \label{eq:tildeY_kurtznotation}
 Y(t)=Y(0)+\sum_{\trackS+y\to \trackS'+y'\in\trackRc}(\trackS'-\trackS)N_{\trackS+y\to \trackS'+y'}\left(\int_0^t \lambda_{\trackS+y\to \trackS'+y'}(Y(s),Z(s))ds\right).
\end{equation}
Then, the following result holds, where we implicitly identify the states of $Y^V$ and $Y$ with the canonical basis of $\RR^{|\track|}$. Note that the assumption that all the components of the solution $Z$ are strictly positive in the time interval $[0,T]$ is made, but this is only a mild restriction to avoid unnecessary technicality, and is always verified under mass-action kinetics as long as $Z(0)\in\RR^d_{>0}$ (see Remark~\ref{rem:positive}). The proof of the result is postponed to Appendix~\ref{sec:proof}, where more precise bounds are given.

\begin{theorem}\label{cor:fddconv}
 Assume that Assumption \ref{ass:classical_limit} holds. Furthermore, assume that the random variables $X^V(0)/V$ converge in probability to some $z^*\in\RR^d_{>0}$ as $V$ goes to infinity, and let $Z(0)=z^*$. Assume that the solution $Z$ to \eqref{eq:drn} with $Z(0)=z^*$ exists over the interval $[0,T]$ and that
 \begin{equation*}
     m=\min_{\substack{i=1,2,\dots,d\\u\in[0,T]}}Z_i(u)>0.
 \end{equation*}
 Finally, assume that $Y^V(0)=Y(0)$ for all positive integers $V$. Then
 \begin{equation}\label{eq:exp_going_0}
 \lim_{V\to\infty}\sup_{t\in[0,T]}P\left(Y^V(t)\neq Y(t)\right)=\lim_{V\to\infty}\sup_{t\in[0,T]}E\left[\|Y^V(t)-Y(t)\|_\infty\right]=0. 
 \end{equation}
\end{theorem}

\begin{remark}\label{rem:positive}
 If we consider mass-action kinetics, then the deterministic solutions never touch the boundaries, provided that the initial condition is strictly positive \cite{S:2001}. In this case, the existence of $m$ as assumed in Theorem~\ref{cor:fddconv} is then guaranteed by $z^*\in\RR^d_{>0}$.
\end{remark}

\begin{remark}\label{rem:fddconv}
 Theorem~\ref{cor:fddconv} implies finite dimensional distribution convergence of $Y^V$ to $Y$ in the following sense: for all $0\leq t_1< t_2<\dots<t_n\leq T$ we have
 \begin{equation*}
  P\left(\max_{1\leq i\leq n} \|Y^V(t_i)-Y(t_i)\|_\infty >0\right)\leq \sum_{i=1}^n P\left(\|Y^V(t_i)-Y(t_i)\|_\infty >0\right),
 \end{equation*}
 and the latter tends to 0 as $V$ tends to $\infty$, under the conditions of Theorem~\ref{cor:fddconv}.
\end{remark}

Some simulations of the process $Y$ are proposed in Figure~\ref{fig:singletraj} for the case of the SIS model \eqref{eq:SIS}.
 \begin{figure}
     \centering
     \includegraphics[width=\textwidth]{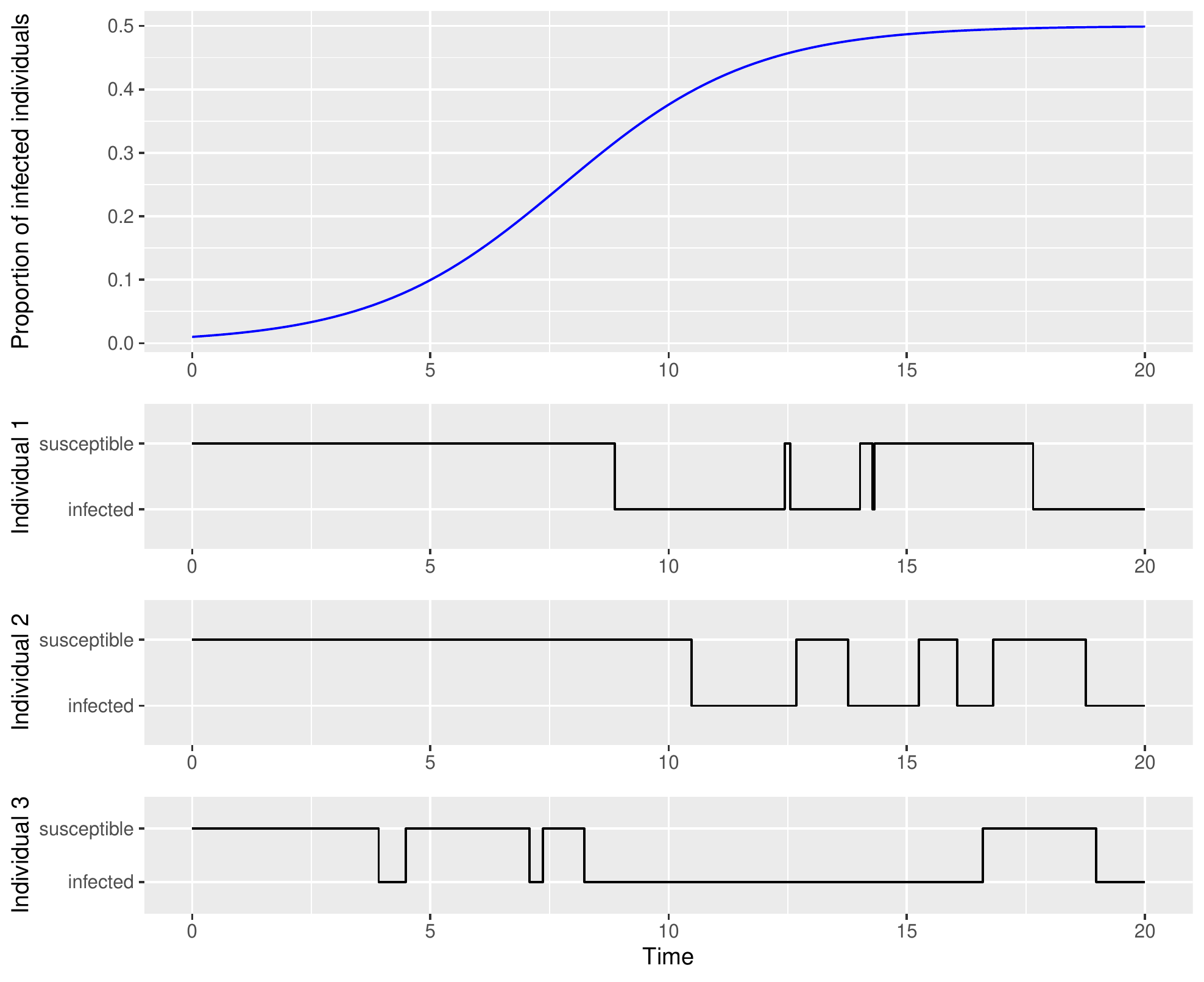}
     \caption{\textbf{The process $Y$ in SIS model.} Consider the model \eqref{eq:SIS}, and let $Y$ be as in \eqref{eq:tildeY_kurtznotation}. The first panel shows the concentration of infected individuals $Z_I$ according to the deterministic solution to \eqref{eq:drn} with $Z_S(0)=0.99$ and $Z_I(0)=0.01$. Mass-action kinetics is assumed, with the rate constants of $S+I\to2I$ and $I\to S$ being 1 and $0.5$, respectively. According to \eqref{eq:tildeY_kurtznotation}, $Z_I$ determines the rate at which the single-individual process $Y$ turns from 'susceptible' to 'infected'. The last three panels show independent realizations of $Y$. The times in the x-axes of the four panels are aligned.}
     \label{fig:singletraj}
 \end{figure}
We conclude this section with the following result, concerning the convergence of $Y^V$ to $Y$ as processes with sample paths in $D_{\track}[0,T]$. We note how this result is necessary for the convergence of continuous functionals of $D_{\track}[0,T]$, as highlighted in Section~\ref{sec:applications_single}.

\begin{theorem}\label{thm:weak_conv}
 Assume that Assumption \ref{ass:classical_limit} holds. Furthermore, assume that the random variables $X^V(0)/V$ converge weakly to a constant $z^*$ as $V$ goes to infinity, and let $Z(0)=z^*$. Assume that the solution $Z$ to \eqref{eq:drn} with $Z(0)=z^*$ exists over the interval $[0,T]$ and that
 \begin{equation*}
     m=\min_{\substack{S\in\Sp\\u\in[0,T]}}Z_S(u)>0.
 \end{equation*}
 Finally, assume that $Y^V(0)=Y(0)$ for all positive integers $V$. Then $Y^V$ converges in probability to $Y$ as processes with sample paths in $D_{\track}[0,T]$ (where we identify $\track$ with the elements of the canonical basis of $\RR^{|\track|}$ and embed it with the metric $\|\cdot\|_\infty$, or any equivalent one).
\end{theorem}

The proof is given in Appendix~\ref{sec:proof}.

\subsection{Applications of Theorem~\ref{thm:weak_conv}}\label{sec:applications_single}

The convergence of Theorem~\ref{thm:weak_conv} allows us to state convergence in probability of $f(Y^V)$ to $f(Y)$, where $f\colon D_{\track}[0,T]\to \RR$ is a functional that is continuous with respect to the Skorokhod $J_1$ topology. Classical examples are $f(x)=\sup_{t\in[0,T]}\|x(t)\|_\infty$, $f(x)=\int_0^T \phi(x(s))ds$ for some continuous function $\phi$, or $f(x)=\sup_{t\in[0,T]} (x(t)-x(t-))$ where $x(t-)=\lim_{h\uparrow t}x(h)$ (see for example \cite[Chapter 3]{EK:1986}). More concretely, a functional we may want to consider is the number of times an individual gets infected in the interval $[0,T]$, assuming the model of equation \eqref{eq:SIS} is in place. We denote this functional by $\psi$. Note that the convergence of $X^V/V$ to its deterministic fluid limit, as stated in Theorem~\ref{thm:classical}, does not give any mean of inferring the distribution of $\psi(Y^V)$. However, knowing that $\psi(Y^V)$ converges in probability to $\psi(Y)$, if $V$ is large enough we can approximate the distribution of the former by the distribution of the latter. Obtaining an estimate of the distribution of $\psi(Y)$ only requires the simulation of enough independent copies of $Y$, whose jump rates are deterministic and therefore do not require a simulation of $X^V$ to be computed, as opposed to the much more expensive strategy of simulating multiple independent trajectories of $(Y^V,X^V)$ via the Gillespie algorithm (which is especially cumbersome for large values of $V$). The empirical distributions obtained with he two strategies are compared in Figure~\ref{fig:ninf}.
 \begin{figure}
     \centering
     \includegraphics[width=\textwidth]{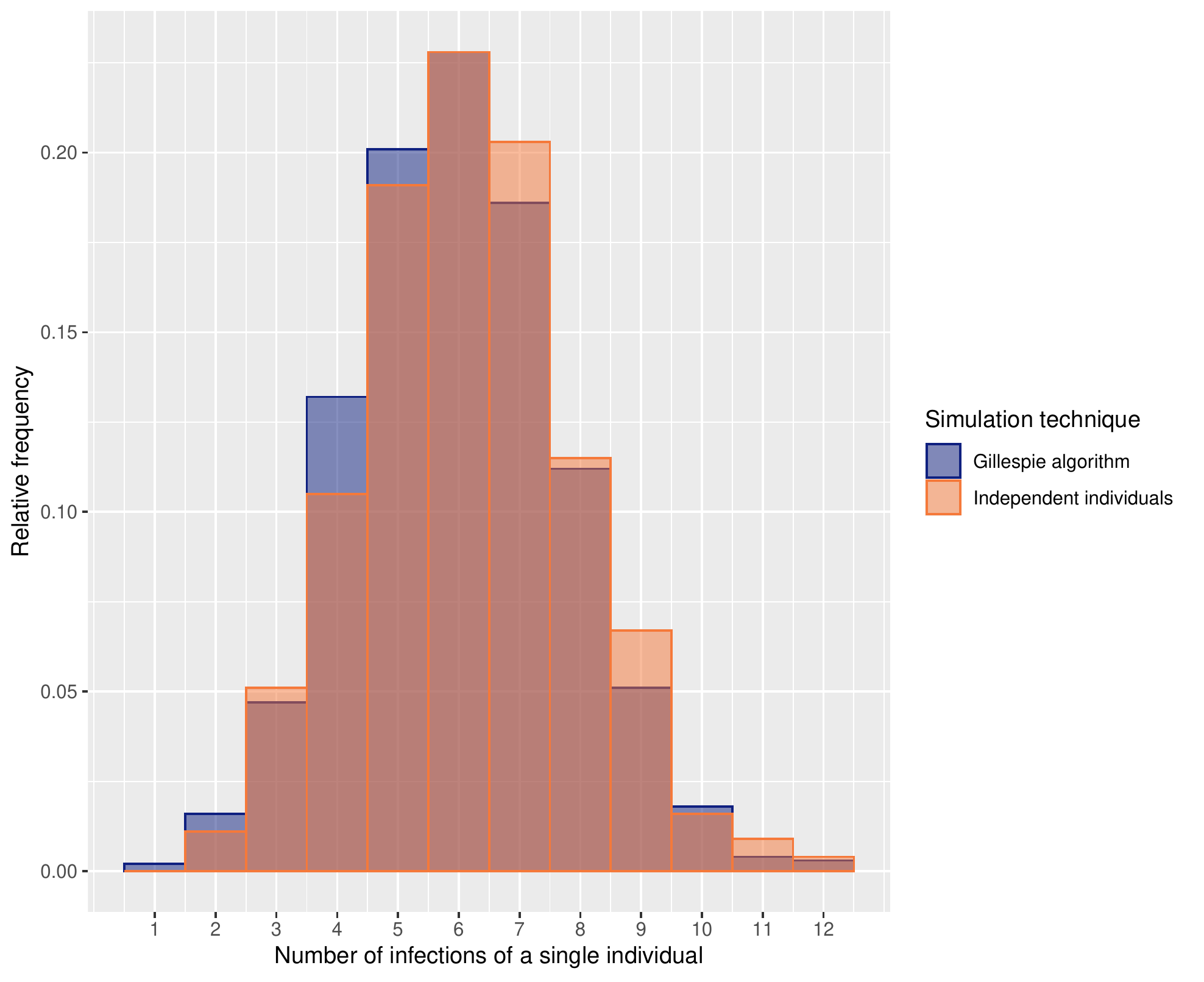}
     \caption{\textbf{Empirical distribution of number of infections in SIS model.} Consider the model \eqref{eq:SIS}, and let $\psi$ be the number of infections a randomly selected individual undergoes up to time $T$. The empirical distributions of $\psi(Y^V)$ and $\psi(Y)$ are compared, the former obtained by the simulation of 1,000 independent copies of $(Y^V,X^V)$ via the Gillespie algorithm (applied to the formulation in terms of usual stochastic reaction networks discussed in Section~\ref{sec:cast}), and the latter obtained via the simulation of 1,000 copies of $Y$. Here, $V=1,000$ and the initial portion of infected individuals is $1\%$ (so we are initially close to the boundary and we may expect some minor discrepancy between $X^V/V$ and its deterministic limit $Z$, see also Figure~\ref{fig:sis}). Mass-action kinetics is assumed, with the rate constants of $S+I\to2I$ and $I\to S$ being 1 and $0.5$, respectively.}
     \label{fig:ninf}
 \end{figure}
 Similarly, we can apply our results to a Michaelis-Menten mechanism. Consider the model
 \begin{equation}\label{eq:futileMM}
  E+S\ce{<=>}C\ce{->}E+P,\quad P\ce{->} S,
 \end{equation}
 where the enzyme activities counterbalances a spontaneous transformation of molecules of type $P$ into molecules of type $S$. To measure the activity level of the enzymes, we may want to study for how long a randomly chosen enzyme molecule is in bound state $C$ up to a given time $T$. Let us call this quantity $\upsilon(Y^V)$. The classical scaling of Theorem~\ref{thm:classical} does not allow for inference of the distribution of $\upsilon(Y^V)$, but Theorem~\ref{thm:weak_conv} ensures that it converges to the distribution of $\upsilon(Y)$ as $V$ tends to $\infty$. Figure~\ref{fig:mm} compares the empirical distributions of $\upsilon(Y^V)$ and $\upsilon(Y)$ obtained by the simulation of $1,000$ independent copies of $(Y^V,X^V)$ and $1,000$ independent copies of $Y$, respectively. For this comparison we chose $V=1,000$.
  \begin{figure}
     \centering
     \includegraphics[width=\textwidth]{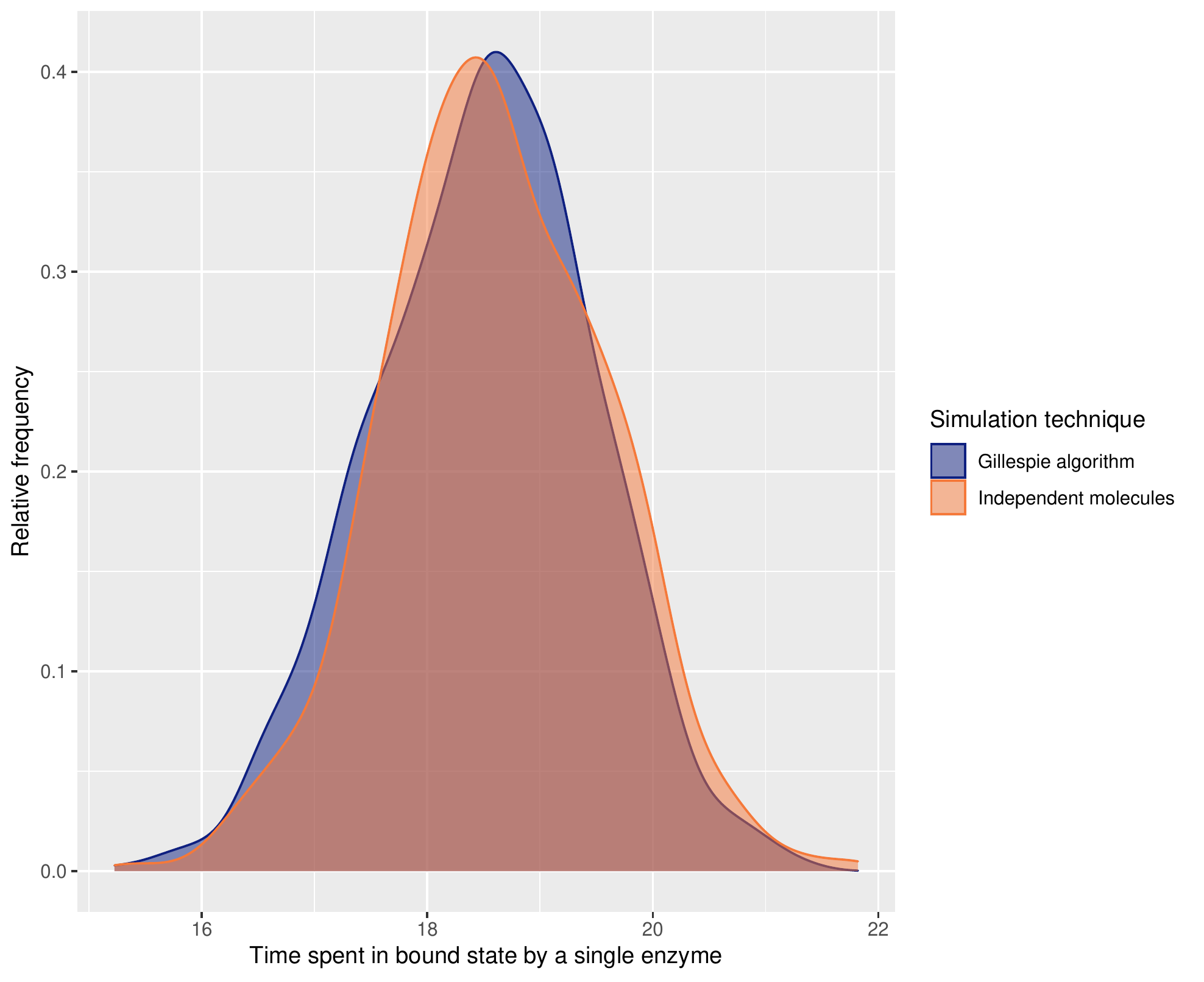}
     \caption{\textbf{Empirical density of time in bound state in Michaelis-Menten model.} Consider the model \eqref{eq:futileMM}, and let $\upsilon$ be the time a randomly selected molecule of enzyme is in bound state $C$ up to time $T$. The empirical distributions of $\upsilon(Y^V)$ and $\upsilon(Y)$ are compared, the former obtained by the simulation of 1,000 independent copies of $(Y^V,X^V)$ via the Gillespie algorithm (applied to the formulation in terms of usual stochastic reaction networks discussed in Section~\ref{sec:cast}), and the latter obtained via the simulation of 1,000 copies of $Y$. Here, $V=1,000$ and $Z(0)=X(0)/V=(0.5, 10, 0.5, 1)$,  where the species are ordered as in $E, S, C, P$. Mass-action kinetics is assumed, with the rate constants of $E+S\to C$, $C\to E+S$, $C\to E+P$, and $P\to S$ being 1, 5, 1, and $0.5$, respectively.}
     \label{fig:mm}
 \end{figure}
 
\subsection{Approximating the  system dynamics with single-molecule trajectories}\label{sec:aggregate}

Let $\tracksub\subseteq\Sp$ be the set of \edit{\emph{tracked species}, i.e.\ the set of chemical species whose molecules (or parts thereof) can be tracked:}
\begin{equation*}
 \tracksub=\{S\in\Sp\,:\,S=\sigma(\trackS)\text{ for some }\trackS\in\track\setminus\{\Delta\}\}.
\end{equation*}
Moreover, let $\proj\colon\RR^d\to\RR^{\vert \tracksub\vert}$ be the projection of the state space onto the coordinates relative to the species in $\tracksub$. The aim of this section is to approximate the dynamics of $\proj(X^V)$ by means of a sum of \emph{independent} processes distributed as in \eqref{eq:tildeY_kurtznotation} (potentially with rescaled dynamics, as shown in the statement of Theorem~\ref{thm:aggregate}). Note that the goal of such an approximation is not to provide a faster simulation method than those present in the literature: our goal is to break down the dynamics of several correlated particles into a set of independent single-molecule trajectories which could be simulated simultaneously by a highly parallelizable algorithm. We begin by identifying each status $\trackS\in\track\setminus\{\Delta\}$ with a different \edit{part of the molecules of} the species $\sigma(\trackS)$: $m$ molecules of species $S\in\Sp$ are available at time $t$ if and only if \edit{for all status $\trackS$ with $\sigma(\trackS)=S$ the quantity of the tracked molecules in status $\trackS$} is $m$ at time $t$. Under this assumption, clearly the process $X^V$ can be expressed in terms of \edit{the status changes of its tracked molecules}, which are typically not independent of each other. We further restrict ourselves to models that are \emph{sub-conservative} with respect to the \edit{tracked molecules}. This means that while \edit{a tracked molecule} can potentially be degraded \edit{(by changing its status to $\Delta$)}, their total mass never increases. Equivalently, we assume that each time a \edit{tracked molecule} is created it is by transformation of another \edit{molecule}. We assume sub-conservativeness \edit{for simplicity: we want to consider independent single-molecule fates, whose} agglomeration is still able to approximately describe the dynamics of the whole system. If we allowed for mass creation, we would need to introduce new molecules over time and track them. Defining the molecule creation times over a finite interval of time independently on each other is technically possible if the creation rate \edit{changes deterministically}: it is sufficient to first simulate a Poisson random variable counting the total number of new molecules in the finite time interval, then consider each creation time as independent of the others with probability density proportional to the deterministic creation rate. However, this procedure requires the introduction of further notation and for the sake of clarity we decided to only present the simpler case of sub-conservative models (with respect to the status).
\begin{assumption}\label{ass:subconservative}
 Let $(Y^V, X^V)$ be a family of tracking stochastic reaction systems. We assume that for each reaction $y\to y'\in\Rc$ and for each $\trackS'\in\track\setminus\{\Delta\}$
 \begin{equation*}
     \sum_{\trackS, \in\track\setminus\{\Delta\}}y_{\sigma(\trackS)}p_{y\to y'}(\trackS, \trackS')=y'_{\sigma(\trackS')}
 \end{equation*}
\end{assumption}
For all $S\in\tracksub, \trackS\in\track\setminus\{\Delta\}$ define
 \begin{equation*}
  \sigma^{-1}(S)=\{\trackS'\in\track\,:\,\sigma(\trackS')=S\}\quad\text{and}\quad\alpha(S)=\#\sigma^{-1}(S)
 \end{equation*}
 The sub-conservation of the model with respect to the \edit{tracked molecules} is formally stated as follows.
 
\begin{lemma}\label{lem:subconservative}
 Let $(Y^V, X^V)$ be a family of tracking stochastic reaction systems satisfying Assumption~\ref{ass:subconservative}. Then, for all $V\in\ZZ_{\geq1}$ and for all $t\in\RR_{>0}$
 \begin{equation}\label{eq:subconservation}
  \|\pi(X^V(t))\|_1\leq \sum_{S\in\tracksub}\alpha(S)X^V_S(t)\leq \sum_{S\in\tracksub}\alpha(S)X^V_S(0).
 \end{equation}
\end{lemma}
\begin{proof}
 The first inequality of \eqref{eq:subconservation} simply follows from the fact that the quantities $\alpha(S)$ are greater than or equal to 1. For the second inequality, simply note that if a reactions $y\to y'\in\Rc$ occurs at time $t$, then
 \begin{align*}
  \sum_{S\in\tracksub}\alpha(S)X^V_S(t)-\sum_{S\in\tracksub}\alpha(S)X^V_S(t-)&=\sum_{S\in\tracksub}\alpha(S)y'_S-\sum_{S\in\tracksub}\alpha(S)y_S\\
  &\hspace{-40pt}=\sum_{\trackS'\in\track\setminus\{\Delta\}}y'_{\sigma(\trackS')}-\sum_{\trackS\in\track\setminus\{\Delta\}}y_{\sigma(\trackS)}\\
  &\hspace{-40pt}=\sum_{\trackS'\in\track\setminus\{\Delta\}}\sum_{\trackS, \in\track\setminus\{\Delta\}}y_{\sigma(\trackS)}p_{y\to y'}(\trackS, \trackS')-\sum_{\trackS\in\track\setminus\{\Delta\}}y_{\sigma(\trackS)}\\
  &\hspace{-40pt}\leq \sum_{\trackS\in\track\setminus\{\Delta\}}y_{\sigma(\trackS)}-\sum_{\trackS\in\track\setminus\{\Delta\}}y_{\sigma(\trackS)}=0.
 \end{align*}
 Note that in the third equality we used Assumption~\ref{ass:subconservative}, and in the last equality we used
 $$\sum_{\trackS'\in\track\setminus\{\Delta\}}p_{y\to y'}(\trackS, \trackS')\leq 1.$$
 Since the quantity $\sum_{S\in\tracksub}\alpha(S)X^V_S$ is not increasing with the occurrence of a reaction, \eqref{eq:subconservation} is proven.
\end{proof}

The main result of this section is the following one, a more detailed version of which is proven in the Appendix. In particular, in Theorem~\ref{thm:aggrconv} a convergence rate of the order of $e^{-C\sqrt{V}}$ for a positive constant $C$ is proven, provided that the initial conditions of $X^V$ and $\tr{X}^V$ are close enough.

\begin{theorem}\label{thm:aggregate}
 Assume that Assumptions~\ref{ass:classical_limit} and \ref{ass:subconservative} are satisfied, and consider a family of tracking stochastic reaction systems $(Y^V, X^V)$. Assume that $V^{-1}X^V(0)$ converges in distribution to some $z^*\in\RR^d_{>0}$ as $V$ goes to infinity and $E[\pi(X^V(0))]<\infty$ for all $V\in\ZZ_{\geq1}$. \edit{Assume that the solution $Z$ to \eqref{eq:drn} with $Z(0)=z^*$ exists over the interval $[0,T]$.} Let $\tr{X}^V(0)=\floor{Vz^*}$ and define the process $\tr{X}^V$ by
 \begin{equation}\label{eq:definition_trX}
  \tr{X}^V(t)=\sum_{\trackS\in\track\setminus\{\Delta\}}\sum_{i=1}^{\tr{X}^V_{\sigma(\trackS)}(0)}\frac{\sigma(Y^{\trackS, i}(t))}{\alpha(\sigma(Y^{\trackS, i}(t)))},
 \end{equation}
 where the processes $(Y^{\trackS, i})_{\trackS\in\track\setminus\{\Delta\}, i\in\ZZ_{\geq1}}$ are independent and satisfy
 \begin{equation*}
  Y^{\trackS, i}(t)=\trackS+\sum_{\trackS'+y\to \trackS''+y'\in\trackRc}(\trackS''-\trackS')N^{\trackS, i}_{\trackS'+y\to \trackS''+y'}\left(\int_0^t\lambda_{\trackS'+y\to \trackS''+y'}(Y^{\trackS,i}(u),Z(u))du\right),
 \end{equation*}
 for a family of independent, identically distributed unit-rate Poisson processes $\{N^{\trackS, i}_r\}_{\trackS\in\track\setminus\{\Delta\}, i\in\ZZ_{\geq1}, r\in\trackRc}$. Then,
 \begin{equation*}
  \lim_{V\to\infty} E\left[\sup_{0\leq s\leq t}\left\|\frac{\proj(X^V(t))}{V}-\frac{\tr{X}^V(t)}{V}\right\|\right]=0.
 \end{equation*}
\end{theorem}

Note that in the definition of $\tr{X}^V$ above we consider \edit{the number of} independent single-molecule trajectories \edit{to match the number of molecules (or parts thereof) of trackable species that} are in the system at time 0. A natural question is whether a good approximation of the original model $X^V$ can be obtained by considering the agglomeration of less independent single-molecule trajectories. However, a detailed study of the error in this case is out of the scope of the present paper.



\begin{example}\label{ex:SIS2}
  \begin{figure}[b!]
     \centering
     \includegraphics[width=\textwidth]{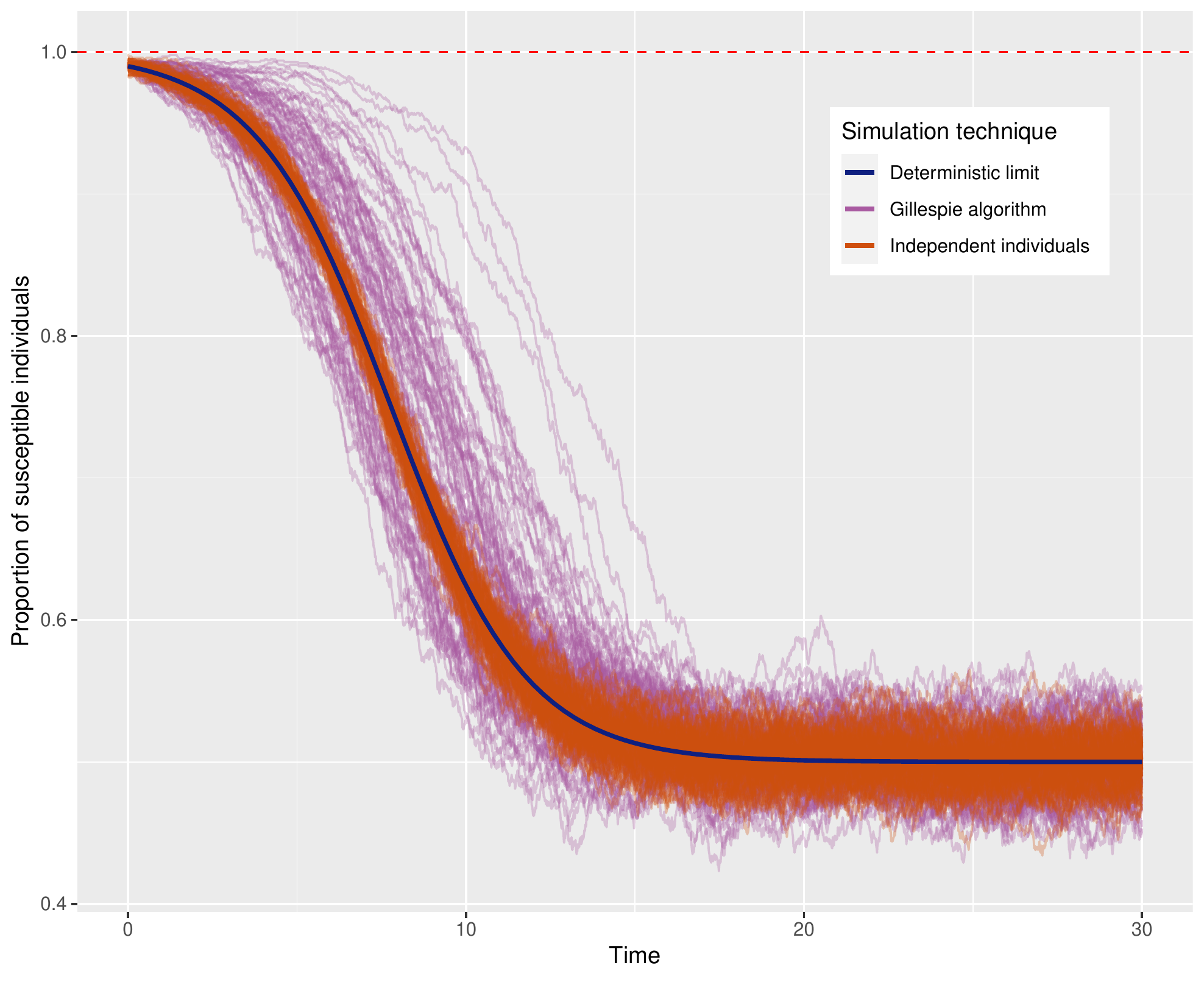}
     \caption{\textbf{Comparison in SIS model.} Comparison of 100 independent trajectories of $X_S^V/V$ and $\tr{X}_S^V/V$, considering the SIS model described in \eqref{eq:SIS}. Here, $X_S^V(0)=0.99 V$, $X_I^V(0)=0.01 V$, and $V=1,000$. Mass-action kinetics is assumed, with the rate constants of $S+I\to2I$ and $I\to S$ being 1 and $0.5$, respectively.}
     \label{fig:sis}
 \end{figure}
 Consider the SIS model of equation \eqref{eq:SIS}. We assume $X_S^V(0)=0.99 V$ and $X_I^V(0)=0.01 V$, and let $V=1,000$. We wish to approximate the number of susceptible individuals by
  \begin{equation*}
     \frac{X_S^V(t)}{V} \approx \frac{\tr{X}_S^V(t)}{V}.
 \end{equation*} 
 In order to test the performance of the above approximation, we simulate 100 independent copies of $X^V$ and $\tr{X}$, and plot them against each other in Figure~\ref{fig:sis}. It is perhaps not surprising to note a higher variance for the trajectories of $X^V$ with respect of those of $\tr{X}^V$: the former is the result of several single-molecule trajectories that are naturally correlated with each other, specifically the rate at which a single molecule changes state is stochastic and given by the current state of all the other molecules. In the approximation, the dynamics of the single tracked molecules are independent and \edit{their rates of transitions between states are completely determined by the deterministic solution $Z$}, which leads to fewer stochastic fluctuations. However, we do observe a discrepancy between the two models only at the beginning of the trajectories, when the number of infected individuals is rather low (only 10 individuals in the initial condition) and the deterministic approximation given by Theorem~\ref{thm:classical} is perhaps not yet accurate enough. As a matter of fact, Figure~\ref{fig:sis2} shows that the difference in variance is considerably reduced if the initial counts of infected individuals is increased to 100.
   \begin{figure}
     \centering
     \includegraphics[width=\textwidth]{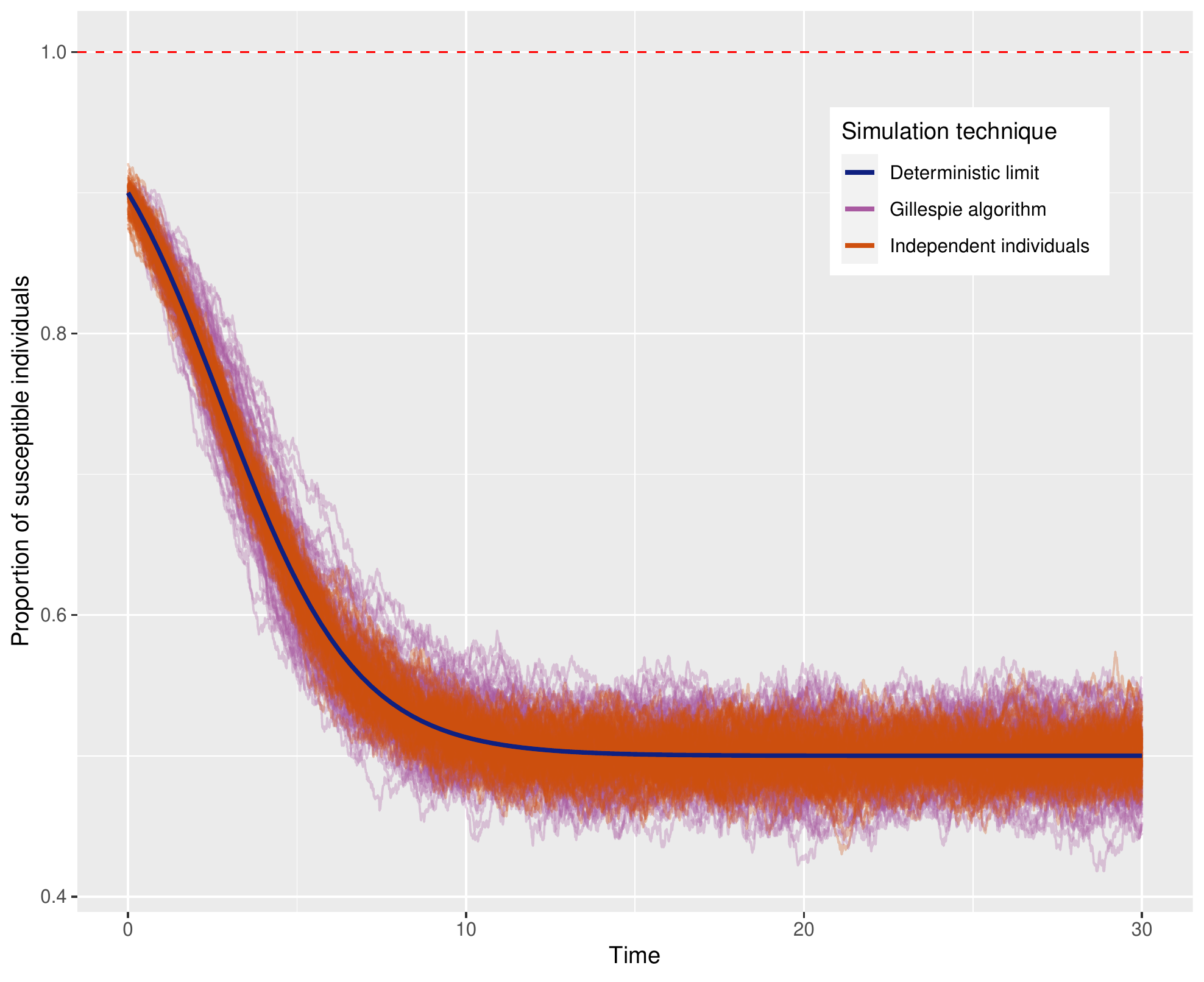}
     \caption{\textbf{Comparison in SIS model.} Comparison of 100 independent trajectories of $X_S^V/V$ and $\tr{X}_S^V/V$, considering the SIS model described in \eqref{eq:SIS}. Here, $X_S^V(0)=0.9 V$, $X_I^V(0)=0.1 V$, and $V=1,000$. Mass-action kinetics is assumed, with the rate constants of $S+I\to2I$ and $I\to S$ being 1 and $0.5$, respectively.}
     \label{fig:sis2}
 \end{figure}
 \end{example}
 We are interested in bounding
 \begin{equation}\label{eq:tobebounded}
     P\left(\sup_{0\leq t\leq T}\left|\frac{X_S^V(t)}{V}-\frac{\tr{X}_S^V(t)}{V}\right|>\varepsilon\right),
 \end{equation}
 for a fixed $\varepsilon\in\RR_{>0}$. Assume mass-action kinetics and let $\kappa_1$ and $\kappa_2$ be the rate constants of $S+I\to 2I$ and $I\to S$, respectively. Moreover, assume for simplicity that $X^V(0)=\tr{X}^V(0)=VZ(0)$ and $X_S^V(0)+X_I^V(0)=V$. Since the total number of individual is conserved, for all $0\leq t\leq T$ we have $X_S^V(t)+X_I^V(t)=V$. By superposition there exist two independent unit-rate Poisson processes $\tr{N}_{S+I\to 2I}$ and $\tr{N}_{I\to S}$ such that for all $0\leq t\leq T$ and for a fixed $V$ we have (with a simplified notation that does not take into account the initial values of the independent single individual trajectories)
 \begin{align*}
     \tr{N}_{S+I\to 2I}\left(\int_0^t \kappa_1 \tr{X}_S^V(u)Z_I(u)du\right)&=\sum_{i=1}^V N^i_{\tr{S}+S+I\to \tr{I}+2I}\left(\int_0^t\mathbbm{1}_{\{\tr{S}\}}(Y^i(u))Z_I(u)du\right)\\
     \tr{N}_{I\to S}\left(\int_0^t \kappa_2 \tr{X}_I^V(u)du\right)&=\sum_{i=1}^V N^i_{\tr{I}+I\to \tr{S}+S}\left(\int_0^t\mathbbm{1}_{\{\tr{I}\}}(Y^i(u))du\right).
 \end{align*}
 Then,
 \begin{multline*}
     \left|\frac{X_S^V(t)}{V}-\frac{\tr{X}_S^V(t)}{V}\right|\leq \Delta(t)+\frac{1}{V}\int_0^t \kappa_1 X_S^V(u)\left|\frac{X_I^V(u)}{V}-Z_I(u)\right|du\\
     +\int_0^t \kappa_1 \left|\frac{X_S^V(u)}{V}-\frac{\tr{X}_S^V(u)}{V}\right|Z_I(u)du+\int_0^t \kappa_2 \left|\frac{X_I^V(u)}{V}-\frac{\tr{X}_I^V(u)}{V}\right|du,
 \end{multline*}
 where
 \begin{align*}
     \Delta(t)=&\frac{1}{V}\left|N_{S+I\to 2I}\left(\int_0^t \frac{\kappa_1}{V} X_S^V(u)X_I^V(u)du\right)-\int_0^t \frac{\kappa_1}{V} X_S^V(u)X_I^V(u)du\right|\\
     &+\frac{1}{V}\left|N_{I\to S}\left(\int_0^t \kappa_2 X_I^V(u)du\right)-\int_0^t \kappa_2 X_I^V(u)du\right|\\
     &+\frac{1}{V}\left|\tr{N}_{S+I\to 2I}\left(\int_0^t \kappa_1 \tr{X}_S^V(u)Z_I(u)du\right)-\int_0^t \kappa_1 \tr{X}_S^V(u)Z_I(u)du\right|\\
     &+\frac{1}{V}\left|\tr{N}_{I\to S}\left(\int_0^t \kappa_2 \tr{X}^V_I(u)du\right)-\int_0^t \kappa_2 \tr{X}^V_I(u)du\right|.
 \end{align*}
 Using $X^V_I(t)=V-X^V_I(t)$ and $Z_I(t)\leq 1$ for all $0\leq t\leq T$ we obtain
 \begin{multline*}
     \left|\frac{X_S^V(t)}{V}-\frac{\tr{X}_S^V(t)}{V}\right|\leq \Delta(t)+\int_0^t\kappa_1\left|\frac{X_I^V(u)}{V}-Z_I(u)\right|du\\
     +\int_0^t(\kappa_1+\kappa_2)\left|\frac{X_S^V(u)}{V}-\frac{\tr{X}_S^V(u)}{V}\right|du.
 \end{multline*}
 By taking the supremum on $0\leq t\leq T$ on both sides and by applying the Gronwall inequality, we have
 \begin{equation*}
     \sup_{0\leq t\leq T}\left|\frac{X_S^V(t)}{V}-\frac{\tr{X}_S^V(t)}{V}\right|\leq \left(\sup_{0\leq t\leq T}\Delta(t)+\kappa_1T\sup_{0\leq t\leq T}\left|\frac{X_I^V(u)}{V}-Z_I(u)\right|\right)e^{(\kappa_1+\kappa_2)T}.
 \end{equation*}
 For notational convenience, let $\nu=\varepsilon e^{-(\kappa_1+\kappa_2)T}$. Hence, \eqref{eq:tobebounded} is smaller than
 \begin{equation}\label{eq:intermediate}
     P\left(\sup_{0\leq t\leq T}\Delta(t)>\frac{\nu}{2}\right)+P\left(\sup_{0\leq t\leq T}\left|\frac{X_I^V(u)}{V}-Z_I(u)\right|>\frac{\nu}{2\kappa_1 T}\right).
 \end{equation}
 By noting that $P(\sup_{0\leq t\leq T}\Delta(t)>\nu/2)$ is smaller than
 \begin{align*}
     &P\left(\sup_{0\leq t\leq T}\frac{1}{V}\left|N_{S+I\to 2I}\left(\int_0^t \frac{\kappa_1}{V} X_S^V(u)X_I^V(u)du\right)-\int_0^t \frac{\kappa_1}{V} X_S^V(u)X_I^V(u)du\right|>\frac{\nu}{8}\right)\\
     &\quad+P\left(\sup_{0\leq t\leq T}\frac{1}{V}\left|N_{I\to S}\left(\int_0^t \kappa_2 X_I^V(u)du\right)-\int_0^t \kappa_2 X_I^V(u)du\right|>\frac{\nu}{8}\right)\\
     &\quad+P\left(\sup_{0\leq t\leq T}\frac{1}{V}\left|\tr{N}_{S+I\to 2I}\left(\int_0^t \kappa_1 \tr{X}_S^V(u)Z_I(u)du\right)-\int_0^t \kappa_1 \tr{X}_S^V(u)Z_I(u)du\right|>\frac{\nu}{8}\right)\\
     &\quad+P\left(\sup_{0\leq t\leq T}\frac{1}{V}\left|\tr{N}_{I\to S}\left(\int_0^t \kappa_2 \tr{X}^V_I(u)du\right)-\int_0^t \kappa_2 \tr{X}^V_I(u)du\right|>\frac{\nu}{8}\right),
 \end{align*}
 we obtain that \eqref{eq:intermediate} is smaller than
 \begin{multline*}
     12\exp\left(\frac{\kappa_1 e T}{2}-\frac{\nu}{24}\sqrt{V}\right)+
     12\exp\left(\frac{\kappa_2 e T}{2}-\frac{\nu}{24}\sqrt{V}\right)\\
     +6\exp\left(\frac{\kappa_1 e T}{2}\left(1+\frac{\nu}{\kappa_1 T}\right)^2+\frac{\kappa_2 e T}{2}\left(1+\frac{\nu}{\kappa_1 T}\right)-\frac{\nu}{12\kappa_1 T}e^{-T( \kappa_1-\kappa_2)-\nu}\sqrt{V}\right)
 \end{multline*}
 by Lemma~\ref{lem:cpp_classical} and Theorem~\ref{thm:estimate_p} (for the special case of the SIS model, see Example~\ref{ex:SISestimates}). We note that $\exp(h)$ is defined as $e^h$ for all real numbers $h$. It follows that \eqref{eq:tobebounded} tends to 0 as $V$ tends to $\infty$ with the same rate as $e^{-C\sqrt{V}}$ for some positive constant $C$. This is always the case, and bounds for more general models are provided by Theorem~\ref{thm:aggrconv}.
 
\section*{Acknowledgements}
 
 DC was supported by the MIUR grant ‘Dipartimenti di Eccellenza 2018-2022’ (E11G18000350001). GAR was supported by the National Sciences Foundation grant (DMS-1853587).
 
\appendix

\section{Proofs and explicit bounds}\label{sec:proof}

In this section we give proofs for the results stated above, together with more precise bounds on the quantities of interest. To this aim, we first define the following quantities: for all $V\in\ZZ_{\geq1}$ and $\varepsilon\in\RR_{>0}$ let
\begin{equation*}
    \A_{V,\varepsilon,t}=\left\{\sup_{u\in[0,t]}\left\|\frac{X^{V}(u)}{V}-Z(u)\right\|_\infty\leq\varepsilon\right\}\quad\text{and}\quad p^{V,\varepsilon,t}=P(\A^c_{V,\varepsilon,t})=1-P(\A_{V,\varepsilon,t}),
\end{equation*}
where the superscript ``$c$'' denotes the complement. Note that, for any fixed $V$ and $\varepsilon$, the sequence of events $\A_{V,\varepsilon,t}$ is monotone in $t$, and $p^{V,\varepsilon,t}$ is a non-decreasing function of $t$ attaining its maximum for the value $t=T$.

Define the $\ZZ^d_{\geq0}$-valued process $X^{V,\varepsilon}$ on $[0,T]$ in the following way: for any $S\in\Sp$ and any $t\in[0,T]$, let
\begin{equation}\label{eq:minmax}
 X_S^{V,\varepsilon}(t)=\min\{\max\{X_S^V(t), VZ_S(t)-V\varepsilon\}, VZ_S(t)+V\varepsilon\}.
\end{equation}
Hence, by definition for all $t\in\RR_{>0}$
 \begin{equation*}
     \left\|\frac{X^{V,\varepsilon}(t)}{V}-Z(t)\right\|_\infty\leq \varepsilon.
 \end{equation*}
Moreover, define the process $\hat{X}^{V, \varepsilon}$ by
 \begin{equation*}
     \hat{X}^{V,\varepsilon}(t)=X^V(0)+\sum_{y\to y'\in\Rc}(y'-y)N_{y\to y'}\left(\int_0^t \lambda_{y\to y'}^V(X^{V,\varepsilon}(u))du\right)
 \end{equation*}
 for all $t\in[0,T]$, where the processes $N_{y\to y'}$ are the same as in \eqref{eq:X_kurtznotation}. Note that for any $u\in[0,t]$ we have $\mathbbm{1}_{\A_{V,\varepsilon,t}}X^{V,\varepsilon}(u)=\mathbbm{1}_{\A_{V,\varepsilon,t}}X^{V}(u)=\mathbbm{1}_{\A_{V,\varepsilon,t}}\hat{X}^{V,\varepsilon}(u)$. In particular, it follows that
 \begin{align}\label{eq:bound}
     \sup_{0\leq u\leq t}\left\|\frac{X^{V,\varepsilon}(u)}{V}-Z(u)\right\|_\infty&\leq
     \mathbbm{1}_{\A_{V,\varepsilon,t}}\sup_{0\leq u\leq t}\left\|\frac{\hat{X}^{V,\varepsilon}(u)}{V}-Z(u)\right\|_\infty+
     \mathbbm{1}_{\A^c_{V,\varepsilon,t}}\varepsilon\notag\\
     &\leq \sup_{0\leq u\leq t}\left\|\frac{\hat{X}^{V,\varepsilon}(u)}{V}-Z(u)\right\|_\infty.
 \end{align}
 \edit{The last inequality follows from noting that if $\A^c_{V,\varepsilon,t}$ occurs and if 
 \begin{equation*}
  u^*=\inf\left\{u\in[0,t]\,:\, \left\|\frac{X^{V}(u)}{V}-Z(u)\right\|_\infty\geq\varepsilon\right\},
 \end{equation*}
 then $X^{V,\varepsilon}(u)=X^{V}(u)=\hat{X}^{V,\varepsilon}(u)$ for all $u\in[0,u^*)$ and $\hat{X}^{V,\varepsilon}(u^*)=X^V(u^*)$. Moreover, by the right continuity of $X^{V}$ and $Z$ $u^*$ is in fact a minimum, which implies 
 \begin{equation*}
  \left\|\frac{X^{V}(u^*)}{V}-Z(u^*)\right\|_\infty\geq\varepsilon.
 \end{equation*}
 Hence
 \begin{align*}
  \mathbbm{1}_{\A^c_{V,\varepsilon,t}}\sup_{0\leq u\leq t}\left\|\frac{\hat{X}^{V,\varepsilon}(u)}{V}-Z(u)\right\|_\infty&\geq \mathbbm{1}_{\A^c_{V,\varepsilon,t}}\left\|\frac{\hat{X}^{V,\varepsilon}(u^*)}{V}-Z(u^*)\right\|_\infty\\
  &=\mathbbm{1}_{\A^c_{V,\varepsilon,t}}\left\|\frac{X^{V}(u^*)}{V}-Z(u^*)\right\|_\infty\geq \mathbbm{1}_{\A^c_{V,\varepsilon,t}}\varepsilon.
 \end{align*}
}

For any $t\in [0,T]$ and any $\varepsilon \in\RR_{>0}$ let
 \begin{equation*}
     \Omega_1^{\varepsilon,t}=\{Z(u)+h\,:\,u\in[0,t], h\in\RR^d, \|h\|_\infty\leq \varepsilon\}\cap\RR^d_{\geq0}
 \end{equation*}
be the (one-dimensional) neighbourhood of the solution $Z$ on the interval $[0,t]$ with amplitude $\varepsilon$, intersected with the non-negative orthant. Note that for all $t\in[0,T]$ we have $X^{V,\varepsilon}(t)/V\in \Omega^{\varepsilon, V}_1$. Similarly, let
 \begin{equation*}
     \Omega_2^{\varepsilon,t}=\{(Z(u)+h, Z(u)+h')\,:\,u\in[0,t], h,h'\in\RR^d, \|h\|_\infty\leq \varepsilon, \|h'\|_\infty\leq \varepsilon\}\cap\RR^{2d}_{\geq0}
 \end{equation*}
 be the two-dimensional neighbourhood of the $Z$ restricted to $[0,t]$ with amplitude $\varepsilon$, intersected with the non-negative orthant. 
 
 To conclude, it is convenient to introduce in this section a notation for centered Poisson processes: given a Poisson process $N$, we denote by $\cen{N}$ the process defined by $\cen{N}(t)=N(t)-t$
for all $t\in\RR_{\geq0}$. In order to bound $p^{V,\varepsilon,t}$ from above and prove Theorem~\ref{thm:aggregate} we need the following results concerning centered Poisson processes. For completeness, we provide a proof as we were not able to find it in the literature, even if small variations of Lemma~\ref{lem:cpp_classical} are well-known and obtained as an application of Doob's inequality or Kolmogorov's maximal inequality. 

\begin{lemma}\label{lem:cpp_classical}
 Let $N$ be a Poisson process and let $T,\varepsilon\in\RR_{>0}$. Then, for all $n \in \ZZ_{\geq1}$
 \begin{equation*}
     P\left(\sup_{t\in[0,nT]}\left|\frac{\cen{N}(t)}{n}\right|>\varepsilon\right)\leq 6 \exp\left(\frac{e}{2}T-\frac{\varepsilon \sqrt{n}}{3}\right).
 \end{equation*}
\end{lemma}

\begin{proof}
 For all $j\in\ZZ_{\geq1}$ and all $h\in\RR_{>0}$ define
 \begin{equation}\label{eq:defxi}
     \Xi^{h}_j=\bigcup_{i=0}^{2^j h}\left\{\frac{i}{2^j}\right\}.
 \end{equation}
 Since $\cen{N}$ is almost surely right continuous, we have that for all $n\in\ZZ_{\geq1}$ and all $T\in\RR_{>0}$
 \begin{equation*}
     \sup_{t\in[0,nT]}\left|\frac{\cen{N}(t)}{n}\right|=\lim_{j\to\infty}\max_{t\in\Xi^{nT}_j}\left|\frac{\cen{N}(t)}{n}\right|
 \end{equation*}
 almost surely. Since for all $j\in\ZZ_{\geq1}$ we have $\Xi^{nT}_j\subset \Xi^{nT}_{j+1}$, by continuity of the probability measure we have
 \begin{equation*}
     P\left(\sup_{t\in[0,nT]}\left|\frac{\cen{N}(t)}{n}\right|>\varepsilon\right)= \lim_{j\to\infty} P\left(\max_{t\in \Xi^{nT}_j}\left|\frac{\cen{N}(t)}{n}\right|>\varepsilon\right).
 \end{equation*}
 By Etemadi's inequality we have
 \begin{equation*}
     P\left(\max_{t\in \Xi^{nT}_j}\left|\frac{\cen{N}(t)}{n}\right|>\varepsilon\right)\leq 3 \max_{t\in \Xi^{nT}_j} P\left(\left|\frac{\cen{N}(t)}{n}\right|>\frac{\varepsilon}{3}\right).
 \end{equation*}
 Moreover, for any real $\beta\in(0,1)$ and any real $t\in(0, nT)$ we have
 \begin{align*}
     P\left(\left|\frac{\cen{N}(t)}{n}\right|>\frac{\varepsilon}{3}\right)&\leq P\left(\frac{\cen{N}(t)}{n}>\frac{\varepsilon}{3}\right)+P\left(-\frac{\cen{N}(t)}{n}>\frac{\varepsilon}{3}\right)\\
     &=P\left(e^{\frac{n^\beta\cen{N}(t)}{n}}>e^{\frac{n^\beta\varepsilon}{3}}\right)+P\left(e^{-\frac{n^\beta\cen{N}(t)}{n}}>e^{\frac{n^\beta\varepsilon}{3}}\right)\\
     &\leq 2\exp\left(-\frac{n^\beta\varepsilon}{3}\right) \exp\left(t(e^{n^{\beta-1}}-1-n^{\beta-1})\right)\\
     &\leq 2\exp\left(-\frac{n^\beta\varepsilon}{3}\right) \exp\left(nT\frac{n^{2\beta-2}}{2}e^{n^{\beta -1}}\right),\\
     &\leq 2\exp\left(-\frac{n^\beta\varepsilon}{3}\right) \exp\left(nT\frac{n^{2\beta-2}}{2}e\right),
 \end{align*}
 where the inequality in the third line follows from the Markov's inequality and the known form of the moment generating function of a Poisson random variable, which leads to $E[e^{n^{\beta-1}\cen{N}(t)}]=e^{-n^{\beta-1} t}e^{t(e^{n^{\beta-1}}-1)}$ and $E[e^{-n^{\beta-1}\cen{N}(t)}]=e^{n^{\beta-1} t}e^{t(e^{-n^{\beta-1}}-1)}$. Hence, for all $n\in\ZZ_{\geq1}$ we have that both $E[e^{n^{\beta-1}\cen{N}(t)}]$ and $E[e^{-n^{\beta-1}\cen{N}(t)}]$ are less than or equal to $e^{t(e^{n^{\beta-1}}-1-n^{\beta-1})}$. The inequality in the forth line derives from the Taylor expansion of the exponential function. By choosing $\beta=1/2$ we have
 \begin{equation*}
     P\left(\left|\frac{\cen{N}(t)}{n}\right|>\frac{\varepsilon}{3}\right)\leq 2\exp\left(-\frac{\varepsilon\sqrt{n}}{3}\right) \exp\left(\frac{e}{2}T\right),
 \end{equation*}
 which completes the proof.
\end{proof}

 \subsection{Estimates for $p^{V,\varepsilon,t}$}\label{sec:estimates_p}
  Many papers have focused on quantifying the distance between the process $X^V$ and its fluid limit $Z$. Among these, we list \cite{agazzi2018large, agazzi2022large, kurtz1976limit, kang2014central, anderson2020tier, prodhomme2020strong} with no claim of completeness. Here we use Lemma~\ref{lem:cpp_classical} to show the following upper bound on $p^{V,\varepsilon,t}$. While similar estimates are known in the reaction network community, we give a formal proof of the bound we propose as we could not find it in the literature. Before stating the result, we define the following quantities:
  \begin{align*}
      R&=\max_{y\to y'\in\Rc}\|y'-y\|_\infty,\\
      \Lambda^{\varepsilon,t}_0&=\sup_{z\in\Omega^{\varepsilon,t}_1}\sum_{y\to y'\in\Rc}\lambda_{y\to y'}(z),\quad \Lambda^{\varepsilon,t}_1=\int_0^t \Lambda^{\varepsilon,u}_0 du\\
      L^{\varepsilon,t}_0&=\sup_{\substack{(z,z')\in\Omega^{\varepsilon,t}_2\\ z\neq z'}}\sum_{y\to y'\in\Rc}\frac{|\lambda_{y\to y'}(z)-\lambda_{y\to y'}(z')|}{\|z-z'\|_\infty},\quad L^{\varepsilon,t}_1=\int_0^tL^{\varepsilon,u}_0 du\\
      \delta^{V,\varepsilon, t}_0&=\sup_{z\in\Omega^{\varepsilon, t}_1}\sum_{y\to y'\in\Rc}\left|\frac{\lambda^V_{y\to y'}(\floor{Vz})}{V}-\lambda_{y\to y'}(z)\right|,\quad\delta^{V,\varepsilon, t}_1=\int_0^t \delta^{V,\varepsilon, u}_0du\\
      \eta^{V,\varepsilon,t}(\gamma)&=e^{-L^{2\varepsilon, t}_1}\gamma\varepsilon-\delta^{V,2\varepsilon, t}_1,
  \end{align*}
  where in the last definition $\gamma$ is any real number in $(0,1]$. Note that $\Lambda^{\varepsilon,t}_0$ and $\delta^{V,\varepsilon,t}_0$ are finite for any $t\in[0,T]$, since the solution $Z$ exists up to time $T$ and the functions $\lambda_{y\to y'}$ are locally Lipschitz by Assumption~\ref{ass:classical_limit}. The local Lipschitzianity of the functions $\lambda_{y\to y'}$ also implies that $L^{\varepsilon, t}_0$ is finite for all $\varepsilon\in\RR_{>0}$ and $t\in[0,T]$. It also follows from Assumption~\ref{ass:classical_limit} that $\delta^{V,\varepsilon,t}_0$ tends to zero as $V$ tends to infinity. Furthermore, note that for fixed $V\in\ZZ_{\geq1}$ and $\varepsilon\in\RR_{>0}$, the quantities $\Lambda^{\varepsilon,t}_0, L^{\varepsilon,t}_0$, and $\delta^{V,\varepsilon,t}_0$ are all non-decreasing functions of $t$. As a consequence, for all $t\in[0,T], \varepsilon\in\RR_{>0}$, and $V\in\ZZ_{\geq1}$ we have
  \begin{equation*}
      \Lambda^{\varepsilon,t}_1\leq t\Lambda^{\varepsilon,t}_0,\quad
      L^{\varepsilon,t}_1\leq tL^{\varepsilon,t}_0,\quad\text{and}\quad
      \delta^{V,\varepsilon,t}_1\leq t\delta^{V,\varepsilon,t}_0.
  \end{equation*}
  It follows that for all $t\in[0,T], \varepsilon\in\RR_{>0}$, and $\gamma\in (0,1]$ the quantity $\eta^{V,\varepsilon,t}(\gamma)$ tends to the positive quantity $e^{-L^{2\varepsilon, t}_1}\gamma\varepsilon$ as $V$ tends to infinity. We can now state the following theorem.
  \begin{theorem}\label{thm:estimate_p}
   For any $\varepsilon,t\in\RR_{>0}$, any $\gamma\in(0,1]$, and any $V\in\ZZ_{\geq1}$ large enough such that $\eta^{V,2\varepsilon, t}(\gamma)>0$, we have
   \begin{equation*}
       p^{V,\varepsilon,t}\leq p^{V, (1-\gamma)\varepsilon e^{-L_1^{2\varepsilon,t}},0}+ 6\exp\left(\frac{e}{2}\Lambda^{2\varepsilon,t}_1+\frac{e}{2}\delta_1^{V,2\varepsilon,t}-\frac{1}{3R}\eta^{V,\varepsilon,t}(\gamma)\sqrt{V}\right)
   \end{equation*}
  \end{theorem}
  \begin{proof}
  First, note that
  \begin{align*}
      p^{V,\varepsilon,t}&=P\left(\sup_{u\in[0,t]}\left\|\frac{X^{V}(u)}{V}-Z(u)\right\|_\infty > \varepsilon\right)=P\left(\sup_{u\in[0,t]}\left\|\frac{X^{V,2\varepsilon}(u)}{V}-Z(u)\right\|_\infty>\varepsilon\right)\\
      &=P\left(\sup_{u\in[0,t]}\left\|\frac{\hat{X}^{V,2\varepsilon}(u)}{V}-Z(u)\right\|_\infty>\varepsilon\right).
  \end{align*}
  Moreover, by superposition, for all $V\in\ZZ_{\geq1}$ and all $\varepsilon\in\RR_{>0}$ we can define a unit-rate Poisson process $U^{V, 2\varepsilon}$ coupled with $X^V$ in such a way that for all $t\in\RR_{\geq0}$
  \begin{equation*}
      U^{V,2\varepsilon}\left(\sum_{y\to y'\in\Rc} \int_0^t\lambda^V_{y\to y'}(X^{V,2\varepsilon}(u))du\right)=\sum_{y\to y'\in\Rc} N_{y\to y'}\left(\int_0^t\lambda^V_{y\to y'}(X^{V,2\varepsilon}(u))du\right).
  \end{equation*}
   Hence, by using \eqref{eq:drn} we have
   \begin{align*}
       \left\|\frac{\hat{X}^{V,2\varepsilon}(u)}{V}-Z(u)\right\|_\infty\leq& \left\|\frac{\hat{X}^{V,2\varepsilon}(0)}{V}-Z(0)\right\|_\infty+\frac{R}{V}\left|\sum_{y\to y'\in\Rc} \cen{N}_{y\to y'}\left(\int_0^u\lambda^V_{y\to y'}(X^{V,2\varepsilon}(w))dw\right)\right|\\
       &+\int_0^u\left| \sum_{y\to y'\in\Rc}\left(\frac{\lambda^V_{y\to y'}(X^{V,2\varepsilon}(w))}{V}-\lambda_{y\to y'}\left(\frac{X^{V,2\varepsilon}(w)}{V}\right)\right)dw\right|\\
       &+\int_0^u\left| \sum_{y\to y'}\left(\lambda_{y\to y'}\left(\frac{X^{V,2\varepsilon}(w)}{V}\right)-\lambda_{y\to y'}(Z(w))\right)dw\right|\\
       \leq& \left\|\frac{X^V(0)}{V}-Z(0)\right\|_\infty+\frac{R}{V}\left|\cen{U}^{V,2\varepsilon}\left(\sum_{y\to y'\in\Rc}\int_0^u\lambda^V_{y\to y'}(X^{V,2\varepsilon}(w))dw\right)\right|\\
       &+\delta_1^{V,2\varepsilon,u}+\int_0^u L^{2\varepsilon, w}_0\left\|\frac{X^{V,2\varepsilon}(w)}{V}-Z(w)\right\|_\infty dw
   \end{align*}
  By using \eqref{eq:bound}, by taking the supremum over $[0,t]$ on both sides we obtain
  \begin{align*}
      \sup_{0\leq u\leq t}\left\|\frac{\hat{X}^{V,2\varepsilon}(u)}{V}-Z(u)\right\|_\infty\leq& 
      \left\|\frac{X^V(0)}{V}-Z(0)\right\|_\infty\\
      &\hspace{-15pt}+\frac{R}{V}\sup_{0\leq u\leq t}\left|\cen{U}^{V,2\varepsilon}\left(\sum_{y\to y'\in\Rc}\int_0^u\lambda^V_{y\to y'}(X^{V,2\varepsilon}(w))dw\right)\right|\\
       &\hspace{-15pt}+\delta^{V,2\varepsilon,t}_1+\int_0^t L^{2\varepsilon, u}_0\sup_{0\leq w\leq u}\left\|\frac{\hat{X}^{V,2\varepsilon}(w)}{V}-Z(w)\right\|_\infty du.
  \end{align*}
  By Gronwall's inequality we get
  \begin{align*}
      \sup_{0\leq u\leq t}\left\|\frac{\hat{X}^{V,2\varepsilon}(t)}{V}-Z(t)\right\|_\infty\leq& e^{L^{2\varepsilon, t}_1}\left\|\frac{X^V(0)}{V}-Z(0)\right\|_\infty\\
      &+\frac{Re^{L^{2\varepsilon, t}_1}}{V}\sup_{0\leq u\leq t}\left|\cen{U}^{V,2\varepsilon}\left(\sum_{y\to y'\in\Rc}\int_0^u\lambda^V_{y\to y'}(X^{V,2\varepsilon}(w))dw\right)\right|\\
      &+e^{L^{2\varepsilon, t}_1}\delta^{V,2\varepsilon, t}_1.
  \end{align*}
  By noting that for all $t\in\RR_{\geq0}$
  \begin{equation*}
      \sup_{z\in\Omega^{2\varepsilon,t}_1}\sum_{y\to y'\in\Rc}\frac{\lambda^V_{y\to y'}(\floor{Vz})}{V}\leq\Lambda_0^{2\varepsilon,t}+\delta_0^{V,2\varepsilon,t},
  \end{equation*}
  we get 
  \begin{align*}
   p^{V,\varepsilon,t}\leq& P\left(e^{L^{2\varepsilon, t}_1}\left\|\frac{X^V(0)}{V}-Z(0)\right\|_\infty>(1-\gamma)\varepsilon\right)\\
   &+P\left(Re^{L^{2\varepsilon, t}_1}\sup_{0\leq u\leq V(\Lambda_1^{2\varepsilon,t}+\delta_1^{V,2\varepsilon,t})}\left|\frac{\cen{U}^{V,2\varepsilon}(u)}{V}\right|+e^{L^{2\varepsilon, t}_1}\delta^{V,2\varepsilon, t}_1>\gamma\varepsilon\right)
  \end{align*}
  for any $\gamma$ in $(0,1]$. The proof is concluded by Lemma~\ref{lem:cpp_classical}.
  \end{proof}
  
  \begin{example}\label{ex:SISestimates}
   Consider the SIS reaction network described in \eqref{eq:SIS}. In this case, in accordance with the classical mass-action choice of kinetics we have
   \begin{equation*}
       \lambda^V_{S+I\to 2I}(x)=\frac{1}{V}\kappa_1x_Sx_I\quad\text{and}\quad\lambda^V_{I\to S}(x)=\kappa_2x_I
   \end{equation*}
   for some positive constants $\kappa_1$ and $\kappa_2$. Hence, Assumption~\ref{ass:classical_limit} is satisfied with
   \begin{equation*}
       \lambda_{S+I\to 2I}(z)=\kappa_1z_Sz_I\quad\text{and}\quad\lambda_{I\to S}(z)=\kappa_2z_I.
   \end{equation*}
   The corresponding solution $Z$ exists for all non-negative times $t$, for all initial conditions $Z(0)=z^*$. Moreover, note that the sum of infected and susceptible individuals is kept constant, hence for all $t\in\RR_{>0}$ we have $Z_S(t)+Z_I(t)=z^*_S+z^*_I=\|z^*\|_1$. In this case we can obtain the following rough estimates
   \begin{align*}
      R&=2,\quad\Lambda^{\varepsilon,t}_0\leq (\|z^*\|_1+\varepsilon)[\kappa_1(\|z^*\|_1+\varepsilon)+\kappa_2],\quad
      L^{\varepsilon,t}_0\leq\kappa_1(\|z^*\|_1+\varepsilon)+\kappa_2,\\
      \delta^{V,\varepsilon, t}_0&=0,\quad \eta^{V,\varepsilon,t}\geq \varepsilon e^{-t \kappa_1(\|z^*\|_1+2\varepsilon)+t\kappa_2}.
  \end{align*}
   If we assume $X^V(0)=Vz^*$, then $p^{V,0,0}=0$. It follows from Theorem~\ref{thm:estimate_p} with the choice $\gamma=1$ that in this case
  \begin{equation*}
       p^{V,\varepsilon,t}\leq 6\exp\left(\frac{t}{2}(\|z^*\|_1+2\varepsilon)[\kappa_1(\|z^*\|_1+2\varepsilon)+\kappa_2]-\frac{\varepsilon\sqrt{V}}{6}e^{-t[ \kappa_1(\|z^*\|_1+2\varepsilon)-\kappa_2]}\right),
   \end{equation*}
   where $\exp(h)$ is defined as $e^h$ for all real numbers $h$.
  \end{example}
 \subsection{Proof of Theorem~\ref{cor:fddconv}}

 First of all, we define some quantities that are useful to give specific bounds on our approximation error.  Define
 \begin{align*}
 \tr{\Lambda}^t_0&=\max_{\trackS\in\track}\sum_{\trackS+y\to \trackS'+y'\in\trackRc}\lambda_{\trackS+y\to \trackS'+y'}(\trackS,Z(t)),\\
 \tr{L}^{\varepsilon,t}_0&=\sup_{\substack{(z,z')\in\Omega^{\varepsilon,t}_2\\ z\neq z'}}\max_{\trackS\in\track}\sum_{\trackS+y\to \trackS'+y'\in\trackRc}\frac{|\lambda_{\trackS+y\to \trackS'+y'}(\trackS,z)-\lambda_{\trackS+y\to \trackS'+y'}(\trackS,z')|}{\|z-z'\|_\infty} \\
 \tr{\delta}^{\edit{V},\varepsilon,t}_0&=\sup_{z\in\Omega^{\varepsilon, t}_1}\max_{\trackS\in\track} \sum_{\trackS+y\to \trackS'+y'\in\trackRc}|\lambda^V_{\trackS+y\to \trackS'+y'}(\trackS,\floor{Vz})-\lambda_{\trackS+y\to \trackS'+y'}(\trackS,z)|\\
 \tr{\Lambda}^t_1&=\int_0^t\tr{\Lambda}^u_0 du,\quad\tr{L}^{\varepsilon,t}_1=\int_0^t L^{\varepsilon, u}_0du,\quad\tr{\delta}^{\edit{V},\varepsilon,t}_1=\int_0^t \tr{\delta}_0^{V,\varepsilon, u}du.
 \end{align*}
Note that $\tr{\Lambda}^t_0$ is finite for any $t\in[0,T]$, due to the fact that $Z$ is defined over the whole interval $[0,T]$. Moreover the functions $\lambda_{\trackS+y\to \trackS'+y'}$ are locally Lipschitz on $\RR^d_{>0}$ by Lemma~\ref{lem:convergence_hatlambda}, hence $\tr{L}^{\varepsilon,t}_0$ is finite for all $t\in[0,T]$. Finally, $\tr{\delta}^{V,\varepsilon,t}_0$ is finite for all $t\in[0,T]$ by Lemma~\ref{lem:convergence_hatlambda}. Note that, for fixed $V$ and $\varepsilon$, the quantities $\tr{L}^{\varepsilon,t}_0$ and $\tr{\delta}^{\edit{V},\varepsilon,t}_0$ are non-decreasing functions of $t$. As a consequence, for all $t\in[0,T], \varepsilon\in\RR_{>0}$, and $V\in\ZZ_{\geq1}$ we have
\begin{equation}\label{eq:bounds}
    \tr{\Lambda}^t_1\leq t\tr{\Lambda}^t_0,\quad\tr{L}^{\varepsilon,t}_1\leq t\tr{L}^{\varepsilon,t}_0,\quad\text{and}\quad\tr{\delta}^{\edit{V},\varepsilon,t}_1\leq t\tr{\delta}^{\edit{V},\varepsilon,t}_0.
\end{equation}

 Before proving Theorem~\ref{cor:fddconv} we show the following stronger result.
 
\begin{theorem}\label{thm:expectation_conv}
 Assume that Assumption \ref{ass:classical_limit} holds. Furthermore, assume that the random variables $X^V(0)/V$ converge in probability to a constant $z^*$ as $V$ goes to infinity. Assume that the solution $Z$ to \eqref{eq:drn} with $Z(0)=z^*$ exists over the interval $[0,T]$ and that
 \begin{equation*}
     m=\min_{\substack{S\in\Sp\\u\in[0,T]}}Z_S(u)>0.
 \end{equation*}
 Finally, assume that $Y^V(0)=Y(0)$ for all positive integers $V$. Then, 
 \begin{equation}\label{eq:exp_prob}
 P\left(Y^V(t)\neq Y(t)\right)=E\left[\|Y^V(t)-Y(t)\|_\infty\right].
 \end{equation}
 Moreover, for any $0<\varepsilon<m$
 \begin{equation*}
     \sup_{t\in[0,T]}E\left[\|Y^V(t)-Y(t)\|_\infty\right]\leq p^{V,\varepsilon,T}+(\tr{\delta}^{V,\varepsilon,T}_1+\varepsilon \tr{L}^{\varepsilon,}_1)e^{2\tr{\Lambda}^T_1}.
 \end{equation*}
\end{theorem}
\begin{proof}
 First, note that
 \begin{equation}\label{sigheraihbvilnil}
  \|Y^V(t)-Y(t)\|_\infty=\begin{cases}
                                   1 & \text{if }Y^V(t)\neq Y(t)\\
                                   0 & \text{if }Y^V(t)= Y(t)
                                  \end{cases},
 \end{equation}
 hence \eqref{eq:exp_prob} holds. Consider the process
\begin{equation}\label{eq:def_Y_bar}
 \hat Y^V(t)=Y(0)+\sum_{\trackS+y\to \trackS'+y'\in\trackRc}(\trackS'-\trackS)N_{\trackS+y\to \trackS'+y'}\left(\int_0^t \lambda^V_{\trackS+y\to \trackS'+y'}(\hat Y^V(u),X^{V,\varepsilon}(u))du\right).
\end{equation}
 \edit{Note that if $\trackS'\neq\trackS$ then $\|\trackS'-\trackS\|_\infty=1$. Moreover, for a unit-rate Poisson process $N$, we have
 \begin{equation*}
  |N(t_1)-N(t_2)|=\begin{cases}
                   N(t_1)-N(t_2)&\text{if }t_1\geq t_2\\
                   N(t_2)-N(t_1)&\text{otherwise}
                  \end{cases}.
 \end{equation*}
 In any case, $|N(t_1)-N(t_2)|$ is distributed as $N(|t_1-t_2|)$.} By equations \eqref{eq:tildeY_kurtznotation} and \eqref{eq:def_Y_bar}, using the triangular inequality, we obtain
 \begin{align*}
  E\Big[\|\hat Y^V(t)-Y(t)\|_\infty\Big]&\\
  &\hspace{-90pt}\edit{\leq E\left[ \sum_{\trackS+y\to \trackS'+y'\in\trackRc}\|\trackS'-\trackS\|_\infty \left|N_{\trackS+y\to \trackS'+y'}\left( \int_0^t\lambda^V_{\trackS+y\to \trackS'+y'}(\hat Y^V(u),X^{V,\varepsilon}(u))du-\int_0^t\lambda_{\trackS+y\to \trackS'+y'}(Y(u),Z(u))du\right)\right|\right]}\\
  &\hspace{-90pt}\leq E\left[ \int_0^t \sum_{\trackS+y\to \trackS'+y'\in\trackRc}\left| \lambda^V_{\trackS+y\to \trackS'+y'}(\hat Y^V(u),X^{V,\varepsilon}(u))-\lambda_{\trackS+y\to \trackS'+y'}(Y(u),Z(u))\right|du\right]\\ 
  &\hspace{-90pt}\leq\Upsilon_1+\Upsilon_2+\Upsilon_3
 \end{align*}
 where
 \begin{align*}
  \Upsilon_1&=E\left[\int_0^t\sum_{\trackS+y\to \trackS'+y'\in\trackRc}\left|\lambda^V_{\trackS+y\to \trackS'+y'}(\hat Y^V(u),X^{V,\varepsilon}(u))-\lambda_{\trackS+y\to \trackS'+y'}\left(\hat Y^V(u),\frac{X^{V,\varepsilon}(u)}{V}\right)\right|du\right]\\
  \Upsilon_2&=E\left[ \int_0^t \sum_{\trackS+y\to \trackS'+y'\in\trackRc}\left|\lambda_{\trackS+y\to \trackS'+y'}\left(\hat Y^V(u),\frac{X^{V,\varepsilon}(u)}{V}\right)-\lambda_{\trackS+y\to \trackS'+y'}(\hat Y^V(u),Z(u))\right|du\right]\\
  \Upsilon_3&=E\left[\int_0^t \sum_{\trackS+y\to \trackS'+y'\in\trackRc}\left|\lambda_{\trackS+y\to \trackS'+y'}(\hat Y^V(u),Z(u))-\lambda_{\trackS+y\to \trackS'+y'}(Y(u),Z(u))\right|du\right]
 \end{align*}
 Since for every $\trackS+y\to \trackS'+y'\in\trackRc$ we have
 \begin{align*}
  \lambda^V_{\trackS+y\to \trackS'+y'}(w,x)&=\mathbbm{1}_{\{\trackS\}}(w)\lambda^V_{\trackS+y\to \trackS'+y'}(\trackS,x)\quad\text{for all }x\in\ZZ^d_{\geq0}, w\in\track\\
  \lambda_{\trackS+y\to \trackS'+y'}(w,z)&=\mathbbm{1}_{\{\trackS\}}(w)\lambda^V_{\trackS+y\to \trackS'+y'}(\trackS,z)\quad\text{for all }z\in\RR^d_{\geq0}, w\in \track,
 \end{align*}
 we can write $\Upsilon_1\leq \tr{\delta}^{V,\varepsilon,t}_1$. Similarly, $\Upsilon_2\leq \varepsilon \tr{L}^{\varepsilon,t}_1$. Finally, 
 \begin{align*}
  \Upsilon_3&=E\left[\int_0^t \sum_{\trackS+y\to \trackS'+y'\in\trackRc}\left|\mathbbm{1}_{\{\trackS\}}(\hat Y^V(u))-\mathbbm{1}_{\{\trackS\}}( Y(u))\right|\lambda_{\trackS+y\to \trackS'+y'}(\trackS,Z(u))du\right]\\
  &\leq E\left[\int_0^t \sum_{\trackS\in\track}\left|\mathbbm{1}_{\{\trackS\}}(\hat Y^V(u))-\mathbbm{1}_{\{\trackS\}}(Y(u))\right|\tr{\Lambda}^u_0du\right]\\
  &= \int_0^t 2P\left(Y^V(u)\neq Y(u)\right)\tr{\Lambda}^u_0du=2\int_0^t E\Big[\|\hat Y^V(u)-Y(u)\|_\infty\Big]\tr{\Lambda}^u_0du,
 \end{align*}
 where in the last equality we used \eqref{eq:exp_prob}. In conclusion,
 \begin{equation*}
  E\Big[\|\hat Y^V(t)-Y(t)\|_\infty\Big]\leq (\tr{\delta}^{V,\varepsilon,t}_1+\varepsilon \tr{L}^{\varepsilon,t}_1)+2\int_0^t E\Big[\|\hat Y^V(u)-Y(u)\|_\infty\Big]\tr{\Lambda}^u_0du.
 \end{equation*}
 By the Gronwall inequality we then have
 \begin{equation*}
  E\Big[\|\hat Y^V(t)-Y(t)\|_\infty\Big]\leq (\tr{\delta}^{V,\varepsilon,t}_1+\varepsilon \tr{L}^{\varepsilon,t}_1)e^{2\tr{\Lambda}^t_1}.   
 \end{equation*}
 The result follows by taking the sup over $t\in[0,T]$ on both sides (the quantity on the right-hand side of the inequality is non-decreasing in $t$) and by noting that $\mathbbm{1}_{\A_{V,\varepsilon, T}}\hat Y^V(t)=\mathbbm{1}_{\A_{V,\varepsilon, T}} Y^V(t)$ for all $t\in[0,T]$. Hence,
 \begin{align*}
  \| Y^V(t)-Y(t)\|_\infty&=\| Y^V(t)-Y(t)\|_\infty\mathbbm{1}_{\A^c_{V,\varepsilon, T}}+\|\hat Y^V(t)-Y(t)\|_\infty\mathbbm{1}_{\A_{V,\varepsilon, T}}\\
  &\leq \mathbbm{1}_{\A^c_{V,\varepsilon, T}}+\|\hat Y^V(t)-Y(t)\|_\infty\mathbbm{1}_{\A_{V,\varepsilon, T}}\\
  &\leq \mathbbm{1}_{\A^c_{V,\varepsilon, T}}+\|\hat Y^V(t)-Y(t)\|_\infty.
 \end{align*}
\end{proof}

We are now ready to prove Theorem~\ref{cor:fddconv}

\begin{proof}[Proof of Theorem~\ref{cor:fddconv}]
 It follows from Theorem~\ref{thm:expectation_conv} that $P\left(Y^V(t)\neq Y(t)\right)=E\left[\|Y^V(t)-Y(t)\|_\infty\right]$. Moreover, for any $\varepsilon>0$ we have $\lim_{V\to\infty} p^{V,\varepsilon,T}=0$ by Theorem~\ref{thm:classical}, and $\lim_{V\to\infty}\tr{\delta}^{V,\varepsilon,T}_1=0$ by Lemma~\ref{lem:convergence_hatlambda} and \eqref{eq:bounds}. Hence,
 \begin{equation*}
  \lim_{V\to\infty}\sup_{t\in[0,T]}E\left[\|Y^V(t)-Y(t)\|_\infty\right]\leq \varepsilon \tr{L}^{\varepsilon,T}_1e^{2\tr{\Lambda}^T_1},
 \end{equation*}
 which concludes the proof by the arbitrariness of $\varepsilon>0$ and by the fact that $\tr{L}^{\varepsilon,T}_0$ (hence $\tr{L}^{\varepsilon,T}_1$) is non-decreasing in $\varepsilon$.
\end{proof}

\subsection{Proof of Theorem~\ref{thm:aggregate}}

Similarly to what was done in the previous section, we define the following quantities to give an upper bound for our approximation error. Define
\begin{align*}
    \hat{R}&=\max_{y\to y'\in\Rc}\|\pi(y'-y)\|_\infty,\quad \hat{r}=\max_{\trackS+y\to \trackS'+y'\in\trackRc}\left\|\frac{\sigma(\trackS')}{\alpha(\sigma(\trackS'))}-\frac{\sigma(\trackS)}{\alpha(\sigma(\trackS))}\right\|_\infty,\\
    \hat{\Lambda}_0^t&=\hat{r}\sum_{\trackS+y\to \trackS'+y'\in\trackRc}\lambda_{\trackS+y\to \trackS'+y'}(\trackS,Z(t)),\quad
    \hat{\Lambda}_1^t=\int_0^t\hat{\Lambda}_0^u du,\\
    \hat{\Lambda}_2^t&=\max_{\trackS\in\track\setminus\{\Delta\}}\sum_{\trackS+y\to \trackS'+y'\in\trackRc}\int_0^t \lambda_{\trackS+y\to \trackS'+y'}(\trackS,Z(u))du,\\
    \hat{\Lambda}^{V,\varepsilon,t}_3&=\int_0^t\sup_{z\in\Omega^{\varepsilon, u}_1}\sum_{y\to y'\in\Rc}\frac{\lambda^V_{y\to y'}(\floor{Vz})}{V}du,\\
    \omega^{\varepsilon,t}&=\hat{r}\sup_{\substack{(z,z')\in\Omega_2^{\varepsilon,t}\\ \|z-z'\|_\infty\leq\varepsilon}}\sum_{\trackS+y\to \trackS'+y'\in\trackRc} \left|\lambda_{\trackS+y\to \trackS'+y'}\left(\trackS,z\right)-\lambda_{\trackS+y\to \trackS'+y'}(\trackS,z')\right|,\\
    \zeta^{\varepsilon, t}&=\int_0^t(\|Z(u))\|_\infty+\varepsilon)du.
\end{align*}

Note that $\hat{\Lambda}_0^t$, $\hat{\Lambda}_2^t$, and $\zeta^{\varepsilon, t}$ are finite for any $t\in[0,T]$, because $Z$ is defined over the whole interval $[0,T]$ and the functions $\lambda_{\trackS+y\to \trackS'+y'}$ are continuous on $\RR^d_{>0}$ by Lemma~\ref{lem:convergence_hatlambda}. Lemma~\ref{lem:convergence_hatlambda} also implies that $\omega^{\varepsilon,t}$ is finite for all $t\in [0,T]$ and $\varepsilon\in\RR_{>0}$. Finally, $\hat{\Lambda}^{V,\varepsilon,t}_3$ is finite by Assumption~\ref{ass:classical_limit}. Note that, for fixed $V$ and $\varepsilon$, the quantities $\hat{\Lambda}^{V,\varepsilon,t}_3$, $\omega^{\varepsilon,t}$, and $\zeta^{\varepsilon, t}$ are non-decreasing functions of $t$.

We now state and prove the following result, which immediately implies Theorem~\ref{thm:aggregate}. Note that $\delta^{V,\varepsilon,t}_1$ is as defined in Section~\ref{sec:estimates_p}.

\begin{theorem}\label{thm:aggrconv}
 Consider a family of tracking stochastic reaction systems $(Y^V, X^V)$, and assume that Assumptions~\ref{ass:classical_limit} and \ref{ass:subconservative} are satisfied. Let $z^*\in\RR^d_{>0}$ and $\tr{X}^V(0)=\floor{Vz^*}$. Define the process $\tr{X}^V$ by
 \begin{equation*}
  \tr{X}^V(t)=\sum_{\trackS\in\track\setminus\{\Delta\}}\sum_{i=1}^{\tr{X}^V_{\sigma(\trackS)}(0)}\frac{\sigma(Y^{\trackS, i}(t))}{\alpha(\sigma(Y^{\trackS, i}(t)))},
 \end{equation*}
 where the processes $(Y^{\trackS, i})_{\trackS\in\track\setminus\{\Delta\}, i\in\ZZ_{\geq1}}$ are independent and satisfy
 \begin{equation*}
  Y^{\trackS, i}(t)=\trackS+\sum_{\trackS'+y\to \trackS''+y'\in\trackRc}(\trackS''-\trackS')N^{\trackS, i}_{\trackS'+y\to \trackS''+y'}\left(\int_0^t \lambda_{\trackS'+y\to \trackS''+y'}(Y(u)^{\trackS, i},Z(u))du\right),
 \end{equation*}
 for a family of independent, identically distributed unit-rate Poisson processes $\{N^{\trackS, i}_r\}_{\trackS\in\track\setminus\{\Delta\}, i\in\ZZ_{\geq1}, r\in\trackRc}$. For arbitrary $\nu_1,\nu_2, \nu_3\in\RR_{>0}$ define
 \begin{equation*}
     \nu=e^{\hat{\Lambda}_1^T}\left(\hat{R}\nu_1+\hat{r}\nu_2+\nu_3+\hat{R}\delta^{V,\varepsilon,T}_1+\omega^{\varepsilon,T}\zeta^{\varepsilon, T}\right)
 \end{equation*}
 Then,
 \begin{multline*}
     P\left(\sup_{0\leq t\leq T}\left\|\frac{\proj(X^V(t))}{V}-\frac{\tr{X}^V(t)}{V}\right\|_\infty>\nu\right)\leq 6\exp\left(\frac{e \hat{\Lambda}^{V,\varepsilon,t}_3}{2}-\frac{\nu_1\sqrt{V}}{3}\right)\\+6\exp\left(\frac{e c\hat{\Lambda}_2^t}{2}-\frac{\nu_2\sqrt{V}}{3}\right)
     +P\left(\left\|\frac{\proj(X^V(0))}{V}-\frac{\tr{X}^V(0)}{V}\right\|_\infty>\nu_3\right)+p^{V,\varepsilon,T},
 \end{multline*}
 where $c=\sum_{S\in\Sp} \alpha(S) z^*_S$.
\end{theorem}
\begin{proof}
  By the superposition property of Poisson processes, for all $V\in\ZZ_{\geq1}$ there exist two unit-rate Poisson processes $U_1^V$ and $U_2^V$ such that for all $t\in\RR_{\geq0}$
  \begin{equation*}
     U_1^V\left(\sum_{y\to y'\in\Rc} \int_0^t \lambda^V_{y\to y'}(X^{V,\varepsilon}(u))du\right)=\sum_{y\to y'\in\Rc} N_{y\to y'}\left(\int_0^t \lambda^V_{y\to y'}(X^{V,\varepsilon}(u))du\right)
 \end{equation*}
 and
 \begin{multline*}
     U_2^V\left(\sum_{\trackS\in\track\setminus\{\Delta\}}\sum_{i=1}^{\tr{X}^V_{\sigma(\trackS)}(0)}\sum_{\trackS'+y\to \trackS''+y'\in\trackRc}\int_0^t\lambda_{\trackS'+y\to \trackS''+y'}(Y^{\trackS, i}(u),Z(u))du\right)\\
     =\sum_{\trackS\in\track\setminus\{\Delta\}}\sum_{i=1}^{\tr{X}^V_{\sigma(\trackS)}(0)}\sum_{\trackS'+y\to \trackS''+y'\in\trackRc} N^{\trackS, i}_{\trackS'+y\to \trackS''+y'}\left(\int_0^t \lambda_{\trackS'+y\to \trackS''+y'}(Y^{\trackS, i}(u),Z(u))du\right)
 \end{multline*}
 Note that
 \begin{multline*}
  \tr{X}^V(t)=\tr{X}^V(0)+\sum_{\trackS\in\track\setminus\{\Delta\}}\sum_{\trackS'+y\to \trackS''+y'\in\trackRc}\sum_{i=1}^{\tr{X}^V_{\sigma(\trackS)}(0)}\left(\frac{\sigma(\trackS'')}{\alpha(\sigma(\trackS''))}-\frac{\sigma(\trackS')}{\alpha(\sigma(\trackS'))}\right)\times\\
  \times N^{\trackS, i}_{\trackS'+y\to \trackS''+y'}\left(\int_0^t \lambda_{\trackS'+y\to \trackS''+y'}(Y(u)^{\trackS, i},Z(u))du\right).
 \end{multline*}
 Hence, by triangular inequality,
  \begin{align*}
     \sup_{0\leq u\leq t} \left\|\frac{\proj(\hat{X}^{V,\varepsilon}(u))}{V}-\frac{\tr{X}^V(u)}{V}\right\|_\infty\leq \left\|\frac{\proj(X^V(0))}{V}-\frac{\tr{X}^V(0)}{V}\right\|_\infty+\sum_{i=1}^5\Upsilon_i
  \end{align*}
  where
  \begin{align*}
  \Upsilon_1=&\sup_{0\leq u\leq t}\sum_{y\to y'\in\Rc}\|\proj(y'-y)\|_\infty\frac{1}{V}\left|\cen{N}_{y\to y'}\left(\int_0^u \lambda^V_{y\to y'}(X^{V,\varepsilon}(w))dw\right)\right|\\
  &\leq \frac{\hat{R}}{V}\sup_{0\leq u\leq t}\left|\cen{U}^V_1\left(\sum_{y\to y'\in\Rc}\int_0^u \lambda^V_{y\to y'}(X^{V,\varepsilon}(w))dw\right)\right|\\
  \Upsilon_2=&\sup_{0\leq u\leq t}\sum_{\trackS\in\track\setminus\{\Delta\}}\sum_{\trackS'+y\to \trackS''+y'\in\trackRc}\sum_{i=1}^{\tr{X}^V_{\sigma(\trackS)}(0)}\left\|\frac{\sigma(\trackS'')}{\alpha(\sigma(\trackS''))}-\frac{\sigma(\trackS')}{\alpha(\sigma(\trackS'))}\right\|_\infty\times\\
  &\quad\times\frac{1}{V}\left|\cen{N}^{\trackS, i}_{\trackS'+y\to \trackS''+y'}\left(\int_0^u \lambda_{\trackS'+y\to \trackS''+y'}(Y^{\trackS, i}(w),Z(w))dw\right)\right|\\
  &\leq \frac{\hat{r}}{V} \sup_{0\leq u\leq t} \left|\cen{U}_2^V\left(\sum_{\trackS\in\track\setminus\{\Delta\}}\sum_{\trackS'+y\to \trackS''+y'\in\trackRc}\sum_{i=1}^{\tr{X}^V_{\sigma(\trackS)}(0)}\int_0^u \lambda_{\trackS'+y\to \trackS''+y'}(Y^{\trackS, i}(w),Z(w))dw\right)\right|\\
  \Upsilon_3=&\sup_{0\leq u\leq t}\sum_{y\to y'\in\Rc}\|\proj(y'-y)\|_\infty\int_0^u \left|\frac{\lambda^V_{y\to y'}(X^{V,\varepsilon}(w))}{V}-\lambda_{y\to y'}\left(\frac{X^{V,\varepsilon}(w)}{V}\right)\right|dw\\
  &\leq \hat{R}\delta^{V,\varepsilon, t}_1\\
  \Upsilon_4=&\sup_{0\leq u\leq t}\Bigg\|
  \sum_{y\to y'\in\Rc}\proj(y'-y)\int_0^u \lambda_{y\to y'}\left(\frac{X^{V,\varepsilon}(w)}{V}\right)dw\\
  &\quad-\sum_{\trackS'+y\to\trackS''+y'\in\trackRc}\left(\frac{\sigma(\trackS'')}{\alpha(\sigma(\trackS''))}-\frac{\sigma(\trackS')}{\alpha(\sigma(\trackS'))}\right)\int_0^u \frac{X^{V,\varepsilon}_{\sigma(\trackS')}(w)}{V}\lambda_{\trackS'+y\to \trackS''+y'}(\trackS',Z(w))dw\Bigg\|_\infty
  \end{align*}
  \begin{align*}
  \Upsilon_5=&\sup_{0\leq u\leq t}\Bigg\|\sum_{\trackS'+y\to\trackS''+y'\in\trackRc}\left(\frac{\sigma(\trackS'')}{\alpha(\sigma(\trackS''))}-\frac{\sigma(\trackS')}{\alpha(\sigma(\trackS'))}\right)\int_0^u \frac{X^{V,\varepsilon}_{\sigma(\trackS')}(w)}{V}\lambda_{\trackS'+y\to \trackS''+y'}(\trackS',Z(w))dw\\
  &\quad-\frac{1}{V}\sum_{\trackS\in\track\setminus\{\Delta\}}\sum_{\trackS'+y\to \trackS''+y'\in\trackRc}\sum_{i=1}^{\tr{X}^V_{\sigma(\trackS)}(0)}\left(\frac{\sigma(\trackS'')}{\alpha(\sigma(\trackS''))}-\frac{\sigma(\trackS')}{\alpha(\sigma(\trackS'))}\right)\int_0^u \lambda_{\trackS'+y\to \trackS''+y'}(Y^{\trackS, i}(w),Z(w))dw\Bigg\|_\infty
 \end{align*}
 We first focus on rewriting $\Upsilon_4$ and $\Upsilon_5$. To this aim, first note that by identifying species with canonical vectors of $\RR^d$ as previously done in the paper, we have that for all $y\in\C$
 \begin{equation*}
     \pi(y)=\sum_{S\in\tracksub}y_S S=\sum_{\trackS\in\track\setminus\{\Delta\}}\frac{y_{\sigma(\trackS)}\sigma(\trackS)}{\alpha(\sigma(\trackS))}.
 \end{equation*}
 Hence, for all $y\to y'\in\Rc$
  \begin{align*}
     \pi(y'-y)&=\sum_{\trackS'\in\track\setminus\{\Delta\}}\frac{y_{\sigma(\trackS')}\sigma(\trackS')}{\alpha(\sigma(\trackS'))} -\sum_{\trackS\in\track\setminus\{\Delta\}}\frac{y_{\sigma(\trackS)}\sigma(\trackS)}{\alpha(\sigma(\trackS))}\\
     &=\sum_{\trackS'\in\track\setminus\{\Delta\}}\frac{\sigma(\trackS')}{\alpha(\sigma(\trackS'))}\sum_{\trackS\in\track\setminus\{\Delta\}}y_{\sigma(\trackS)}p_{y\to y'}(\trackS, \trackS') -\sum_{\trackS\in\track\setminus\{\Delta\}}\frac{y_{\sigma(\trackS)}}{\alpha(\sigma(\trackS))} \sigma(\trackS),
 \end{align*}
 where we used Assumption~\ref{ass:subconservative} in the last equality. By recalling that $\sigma(\Delta)=0$ and $\sum_{\trackS'\in\track}p_{y\to y'}(\trackS,\trackS')$ for all $y\to y'\in\Rc$ and $\trackS\in\track$, we further obtain
 \begin{align*}
     \pi(y'-y)=&\sum_{\trackS'\in\track}\frac{\sigma(\trackS')}{\alpha(\sigma(\trackS'))}\sum_{\trackS\in\track\setminus\{\Delta\}}y_{\sigma(\trackS)}p_{y\to y'}(\trackS, \trackS') \\
     &-\sum_{\trackS\in\track\setminus\{\Delta\}}\frac{y_{\sigma(\trackS)}}{\alpha(\sigma(\trackS))} \sigma(\trackS)\sum_{\trackS'\in\track}p_{y\to y'}(\trackS, \trackS')\\
     =&\sum_{\trackS\in\track\setminus\{\Delta\}}\sum_{\trackS'\in\track}\left(\frac{\sigma(\trackS')}{\alpha(\sigma(\trackS'))}-\frac{\sigma(\trackS)}{\alpha(\sigma(\trackS))}\right)y_{\sigma(\trackS)}p_{y\to y'}(\trackS,\trackS').
 \end{align*}
 It follows that
 \begin{align*}
     &\sum_{y\to y'\in\Rc}\proj(y'-y)\int_0^u \lambda_{y\to y'}\left(\frac{X^{V,\varepsilon}(w)}{V}\right)dw\\
     &\;=\sum_{\trackS'+y\to \trackS''+y'\in\trackRc}\left(\frac{\sigma(\trackS'')}{\alpha(\sigma(\trackS''))}-\frac{\sigma(\trackS')}{\alpha(\sigma(\trackS'))}\right)\int_0^u y_{\sigma(\trackS')}p_{y\to y'}(\trackS',\trackS'')\lambda_{y\to y'}\left(\frac{X^{V,\varepsilon}(w)}{V}\right)dw\\
     &\;=\sum_{\trackS'+y\to \trackS''+y'\in\trackRc}\left(\frac{\sigma(\trackS'')}{\alpha(\sigma(\trackS''))}-\frac{\sigma(\trackS')}{\alpha(\sigma(\trackS'))}\right)\int_0^u \frac{X^{V,\varepsilon}_{\sigma(\trackS)}(w)}{V}\lambda_{\trackS'+y\to \trackS''+y'}\left(\trackS',\frac{X^{V,\varepsilon}(w)}{V}\right)dw,
 \end{align*}
 which in turn implies
 \begin{align*}
     \Upsilon_4\leq& \sup_{0\leq u\leq t}\sum_{\trackS'+y\to \trackS''+y'\in\trackRc}\left\|\frac{\sigma(\trackS'')}{\alpha(\sigma(\trackS''))}-\frac{\sigma(\trackS')}{\alpha(\sigma(\trackS'))}\right\|_\infty\times\\
     &\quad\times\int_0^u \frac{X^{V,\varepsilon}_{\sigma(\trackS)}(w)}{V}\left|\lambda_{\trackS'+y\to \trackS''+y'}\left(\trackS',\frac{X^{V,\varepsilon}(w)}{V}\right)-\lambda_{\trackS'+y\to \trackS''+y'}(\trackS',Z(w))\right|dw\\
     \leq& \omega^{\varepsilon,t}\zeta^{\varepsilon,t}.
 \end{align*}
 By summing over the values of the single-molecule trajectories, we also have 
 \begin{equation*}
  \sum_{\trackS\in\track\setminus\{\Delta\}}\sum_{i=1}^{\tr{X}^V_{\sigma(\trackS)}(0)}\lambda_{\trackS'+y\to \trackS''+y'}(Y^{\trackS, i}(w),Z(w))=\tr{X}^V_{\sigma(\trackS\edit{'})}(w)\lambda_{\trackS'+y\to \trackS''+y'}(\trackS\edit{'},Z(w)),
 \end{equation*}
 which implies
  \begin{align*}
     \Upsilon_5\leq& \sup_{0\leq u\leq t} \sum_{\trackS'+y\to \trackS''+y'\in\trackRc}\left\|\frac{\sigma(\trackS'')}{\alpha(\sigma(\trackS''))}-\frac{\sigma(\trackS')}{\alpha(\sigma(\trackS'))}\right\|_\infty\int_0^u \left|\frac{X^{V,\varepsilon}_{\sigma(\trackS\edit{'})}(w)}{V}-\frac{\tr{X}^{V}_{\sigma(\trackS\edit{'})}(w)}{V}\right|\lambda_{\trackS'+y\to \trackS''+y'}(\trackS',Z(w))dw\\
     \leq& \int_0^t\left\|\frac{X^{V,\varepsilon}(u)}{V}-\frac{\tr{X}^{V}(u)}{V}\right\|_\infty \hat{\Lambda}_0^u du\\
     =&\mathbbm{1}_{\A^c_{V,\varepsilon,t}}\int_0^t\left\|\frac{X^{V,\varepsilon}(u)}{V}-\frac{\tr{X}^{V}(u)}{V}\right\|_\infty \hat{\Lambda}_0^u du+ \mathbbm{1}_{\A_{V,\varepsilon,t}}\int_0^t\left\|\frac{\hat{X}^{V,\varepsilon}(u)}{V}-\frac{\tr{X}^{V}(u)}{V}\right\|_\infty \hat{\Lambda}_0^u du.\\
     \leq& \mathbbm{1}_{\A^c_{V,\varepsilon,t}}M^{V,\varepsilon,t}+ \int_0^t\left\|\frac{\hat{X}^{V,\varepsilon}(u)}{V}-\frac{\tr{X}^{V}(u)}{V}\right\|_\infty \hat{\Lambda}_0^u du,
 \end{align*}
 where
 \begin{equation*}
     M^{V,\varepsilon, t}=\int_0^t \left(\|Z(u)\|_\infty+\varepsilon+\sum_{S\in\Sp}\alpha(S)\frac{\tr{X}^V_S(0)}{V}\right)\hat{\Lambda}_0^u du
 \end{equation*}
 is an almost surely finite random variable, non-decreasing in $t$.
 Hence, putting everything together and applying the Gronwall inequality we have that almost surely
 \begin{align*}
      &\sup_{0\leq t\leq T}\left\|\frac{\proj(\hat{X}^{V,\varepsilon}(t))}{V}-\frac{\tr{X}^V(t)}{V}\right\|_\infty\leq e^{\hat{\Lambda}_1^T} \frac{\hat{R}}{V}\sup_{0\leq t\leq T}\left|\cen{U}^V_1\left(\sum_{y\to y'\in\Rc}\int_0^t \lambda^V_{y\to y'}(X^{V,\varepsilon}(u))du\right)\right|\\
     &\;+e^{\hat{\Lambda}_1^T} \frac{\hat{r}}{V}\sup_{0\leq t\leq T}\left|\cen{U}^V_2\left(\sum_{\trackS\in\track\setminus\{\Delta\}}\sum_{\trackS'+y\to \trackS''+y'\in\trackRc}\sum_{i=1}^{\tr{X}^V_{\sigma(\trackS)}(0)}\int_0^t \lambda_{\trackS'+y\to \trackS''+y'}(Y^{\trackS, i}(u),Z(u))du\right)\right|\\
     &\;+e^{\hat{\Lambda}_1^T}\Bigg(\left\|\frac{\proj(X^V(0))}{V}-\frac{\tr{X}^V(0)}{V}\right\|_\infty+\hat{R}\delta^{V,\varepsilon,T}_1+\omega^{\varepsilon,T}\zeta^{\varepsilon,T}+\mathbbm{1}_{\A^c_{V,\varepsilon,T}}M^{V,\varepsilon,T}\Bigg).
 \end{align*}
 Now note that if $A_1, A_2, \dots, A_j$ are random variables and $a_1, a_2, \dots, a_j$ are positive real numbers, then
 \begin{equation*}
     P\left(\sum_{i=1}^j A_i>\sum_{i=1}^j a_i\right)\leq P\left(\bigcup_{i=1}^j (A_i>a_i)\right)\leq \sum_{i=1}^j P(A_i>a_i).
 \end{equation*}
 Hence, if $\nu$ is as in the statement of the theorem and $\nu<\varepsilon$, 
 \begin{align*}
     &P\left(\sup_{0\leq t\leq T}\left\|\frac{\proj(X^V(t))}{V}-\frac{\tr{X}^V(t)}{V}\right\|_\infty>\nu\right)=P\left(\sup_{0\leq t\leq T}\left\|\frac{\proj(\hat{X}^{V,\varepsilon}(t))}{V}-\frac{\tr{X}^V(t)}{V}\right\|_\infty>\nu\right)\\
     &\leq P\left(\frac{1}{V}\sup_{0\leq t\leq T}\left|\cen{U}^V_1\left(\sum_{y\to y'\in\Rc}\int_0^t \lambda^V_{y\to y'}(X^{V,\varepsilon}(u))du\right)\right|>\nu_1\right)\\
     &\;+P\left(\frac{1}{V}\sup_{0\leq t\leq T}\left|\cen{U}^V_2\left(\sum_{\trackS\in\track\setminus\{\Delta\}}\sum_{\trackS'+y\to \trackS''+y'\in\trackRc}\sum_{i=1}^{\tr{X}^V_{\sigma(\trackS)}(0)}\int_0^t \lambda_{\trackS'+y\to \trackS''+y'}(Y^{\trackS, i}(u),Z(u))du\right)\right|>\nu_2\right)\\
     &\;+p^{V,\varepsilon,T}.
 \end{align*}
 Since for all $t\in[0,T]$
  \begin{equation*}
     \int_0^t \lambda^V_{y\to y'}(X^{V,\varepsilon}(u))du\leq V\hat{\Lambda}^{V,\varepsilon,t}_3
 \end{equation*}
 and
 \begin{equation*}
     \sum_{\trackS\in\track\setminus\{\Delta\}}\sum_{\trackS'+y\to \trackS''+y'\in\trackRc}\sum_{i=1}^{\tr{X}^V_{\sigma(\trackS)}(0)}\int_0^t \lambda_{\trackS'+y\to \trackS''+y'}(Y^{\trackS, i}(u),Z(u))du\leq V c \hat{\Lambda}_2^t,
 \end{equation*}
 the proof is concluded by Lemma~\ref{lem:cpp_classical}.
 \end{proof}

 \begin{proof}[Proof of Theorem~\ref{thm:aggregate}]
  Note that by Lemma~\ref{lem:subconservative} and by the fact that $\alpha(S)\geq 1$ for all $S\in\Sp$ in \eqref{eq:definition_trX},
 \begin{align*}
  \left\|\frac{\proj(X^V(h))}{V}-\frac{\tr{X}^V(h)}{V}\right\|_1&\leq \left\|\frac{\proj(X^V(h))}{V}\right\|_1+\left\|\frac{\tr{X}^V(h)}{V}\right\|_1\\
  &\leq\frac{1}{V}\left(\sum_{S\in\tracksub} \alpha(S)\left(X^V_S(0)+\tr{X}^V_S(0)\right)\right).
 \end{align*}
 Under the assumption that both $X^V(0)$ and $\tr{X}^V(0)$ have finite expectation and converge in probability to $z^*$, and by the equivalence of norms in finite dimension, we conclude there exists $M\in\RR_{>0}$ such that
 \begin{equation*}
     \sup_{V\in\ZZ_{\geq1}}E\left[\left\|\frac{\proj(X^V(h))}{V}-\frac{\tr{X}^V(h)}{V}\right\|_\infty\right]\leq M.
 \end{equation*}
 Hence, if $\nu$ is as in Theorem~\ref{thm:aggrconv}, we have that 
 \begin{multline*}
  E\left[\sup_{0\leq t\leq T}\left\|\frac{\proj(X^V(h))}{V}-\frac{\tr{X}^V(h)}{V}\right\|_\infty\right]\leq\nu+ 6Me^{\frac{\hat{\Lambda}^{V,\varepsilon,t}_3}{2}-\frac{\nu_1\sqrt{V}}{3}}\\
  +6Me^{\frac{c\hat{\Lambda}_2^t}{2}-\frac{\nu_2\sqrt{V}}{3}}+MP\left(\left\|\frac{\proj(X^V(0))}{V}-\frac{\tr{X}^V(0)}{V}\right\|_\infty>\nu_3\right)+Mp^{V,\varepsilon,T}.
 \end{multline*}
 The proof is concluded if we can show that for all $T\in\RR_{>0}$ and any arbitrary $\eta>0$, we can fix $\nu_1, \nu_2,\nu_3\in\RR_{>0}$ and $\varepsilon\in(0,m)$ such that $\nu<\eta$ for large enough values of $V$. Indeed, for any fixed $\varepsilon\in(0,m),T\in\RR_{>0}$ the other terms on the right-hand side of the above inequality tend to zero as $V$ goes to infinity. To show that $\nu$ can be made smaller than $\eta$, simply note that $\nu_1, \nu_2,\nu_3$ can be chosen as small as desired among the positive real numbers, $\delta^{V,\varepsilon,T}_1$ tends to zero as $V$ goes to infinity for all fixed $\varepsilon\in(0,m)$ by Assumption~\ref{ass:classical_limit}, and $\omega^{\varepsilon,T}$ tends to zero as $\varepsilon$ tends to zero because the functions $\lambda_{\trackS+y\to \trackS'+y'}$ are locally Lipschitz on $\track\times\RR^d_{>0}$ by Lemma~\ref{lem:convergence_hatlambda}.
\end{proof}
 
\subsection{Proof of Theorem~\ref{thm:weak_conv}}

 Note that under the assumptions of Theorem~\ref{thm:weak_conv}, for all $t\in[0,T]$ $Y^V(t)$ converges in probability to $Y(t)$ by Theorem~\ref{cor:fddconv}. Hence, in order to prove Theorem~\ref{thm:weak_conv}, we need to show relative compactness of $\{Y^V\}$ as a sequence of processes with sample paths in $D_{\track}[0,T]$, and conclude by \cite[Lemma A2.1]{DK:1996}, \edit{stated here for convenience.
 \begin{theorem}[Lemma A2.1 in \cite{DK:1996}]
  Consider a sequence of stochastic processes $\{U^V\}$ with sample paths in $D_{E}[0,T]$ defined on the same probability space. Suppose that $\{U^V\}$ is relatively compact in $D_{E}[0,T]$, (in the sense of convergence in distribution) and that for a dense set $H\subseteq [0,\infty)$, $\{U^V(t)\}$ converges in probability in $E$ for each $t\in H$. Then $\{U^V\}$ converges in probability in $D_{E}[0,T]$.
 \end{theorem}
 }
 To prove relative compactness of $\{Y^V\}$, \edit{we use \cite[Corollary 7.4, Chapter 3]{EK:1986}, which we state here for convenience.
 \begin{theorem}[Corollary 7.4 in Chapter 3 of \cite{EK:1986}]
  Let $(E, r)$ be complete and separable, and let $\{U^V\}$ be a sequence of stochastic processes with sample paths in $D_E[O,T]$. Then $\{U^V\}$ is relatively compact if and only if the following two conditions hold:
  \begin{enumerate}
   \item For every $\varepsilon > 0$ and rational $t>0$, there exists a compact set $\Gamma_{\varepsilon, t}\subseteq E$ such that
   $$\liminf_{V\to\infty} P\left(U^V(t)\in \Gamma_{\varepsilon, t}\right)\geq 1-\varepsilon.$$
   \item For every $\varepsilon > 0$ and $T>0$, there exists $\delta>0$ such that
   $$\limsup_{V\to\infty} P\left(\inf_{\{s_i\}}\max_i \sup_{s,t \in [s_{i-1}, s_i)} r(U^V(s), U^V(t))\geq \varepsilon \right) \leq \varepsilon,$$
   where $\{s_i\}$ ranges over all time sequences of the form $0=s_0<s_1<\dots<s_{n-1}<T\leq s_n$ with $\min_{1\leq i\leq n} (s_i-s_{i-1})>\delta$ and $n\geq 1$.  
  \end{enumerate}
 \end{theorem}
 In our case, the topological space $\track$ with the distance induced by $\|\cdot\|_\infty$ is discrete, complete, and separable. It is also compact, so the first condition in the theorem above is always satisfied. Moreover, if a jump occurs at time $t$ then $\|Y^V(t-)-Y^V(t)\|_\infty=1$.} Let $t^V_i$ with $i\in\ZZ_{\geq1}$ denote the time of the $i$th jump of $Y^V$, let $t^V_0=0$, and let $T^V$ be the time of the last jump of $Y^V$ in $[0,T]$. \edit{ Then, as a direct consequence of the theorem above we can state that the sequence of stochastic processes $\{Y^V\}$ with sample paths in $D_{\track}[O,T]$ is relatively compact if and only if for all $\varepsilon>0$ there exists $\delta>0$ such that
 $$\limsup_{V\to\infty}P\left(\min_{j=1,\dots,T^V}(t^V_j-t^V_{j-1})\leq\delta\right)\leq \varepsilon.$$
 }
 Fix $\delta\in\RR_{>0}$ and for all $j\in\ZZ$ with $-1\leq j\leq T/\delta$ let $N^{V,\delta}_j$ be the number of jumps of $Y^V$ in the interval $[j/\delta, \min\{j/\delta+2\delta, T\}]$. The $N^{V,\delta}_j$ are introduced to control the time between jumps: whenever two jumps occur at times differing for less than $\delta$, there necessarily exists an interval $[j/\delta, \min\{j/\delta+2\delta, T\}]$ with $j\geq0$ containing both of them. Also, whenever the time of a jump is smaller than \edit{$\delta$}, then $N^{V,\delta}_{-1}\geq1$. Hence, for all $\nu\in\RR_{>0}$ with $\nu>m$,
 \begin{align*}
     P\left(\min_{j=1,\dots,T^V}(t^V_j-t^V_{j-1})\leq\delta\right)
     &\leq P\left(N^{V,\delta}_{-1}\geq 1\text{ or }\max_{j=1,\dots,\floor{T/\delta}}N^{V,\delta}_j\geq 2\right)\\
     &\leq P\left(N^{V,\delta}_{-1}\geq 1\right)+\sum_{j=1}^{\floor{T/\delta}} P(N^{V,\delta}_j\geq 2)\\
     &\leq P\left(\sup_{0\leq t\leq T}\left\|\frac{X^V}{V}(t)-Z(t)\right\|_\infty>\nu\right)
     +P(N^{\nu}(\delta)\geq 1)
     +\frac{T}{\delta}P(N^{\nu}(2\delta)\geq 2),
 \end{align*}
 where $N^{\nu}$ is a Poisson process with rate
 \begin{equation*}
  B_\nu=\sup_{N\in\ZZ_{\geq1}}\sup_{z\in\Omega_1^{\nu,T}}\max_{\tr{S}\in\track} \sum_{\tr{S}+y\to\tr{S}'+y'\in\trackRc}\lambda^V_{\tr{S}+y\to\tr{S}'+y'}(\tr{S}, \floor{Vz}),
 \end{equation*}
 which is finite by Lemma~\ref{lem:convergence_hatlambda}. Hence, by Theorem~\ref{thm:classical}
 \begin{equation*}
     \limsup_{V\to\infty}P\left(\min_{j=1,\dots,T^V}(t^V_j-t^V_{j-1})\leq\delta\right)
     \leq (1-e^{-\delta B_\nu})+\frac{T}{\delta}(1-e^{-2\delta B_\nu}-2\delta B_\nu e^{-2\delta B_\nu}),
 \end{equation*}
 which tends to 0 as $\delta$ tends to 0. \edit{The proof is completed.}
\bibliographystyle{plain}
\bibliography{greg_daniele_revised2}
\end{document}